\documentclass{amsart}
\usepackage{multirow}
\usepackage{tabu}
\usepackage{adjustbox}
\usepackage{rotating}
\usepackage{enumerate}
\usepackage{amsmath}
\usepackage{mathrsfs}
\usepackage{booktabs,caption}
\usepackage[flushleft]{threeparttable}
\usepackage{amssymb}
\usepackage{mathtools}
\usepackage{color}
\usepackage{algorithm}
\usepackage{algorithmic}
\usepackage{cite,url}
\usepackage{latexsym}
\usepackage{booktabs}
\usepackage{bm}
\usepackage{graphicx}
\usepackage[flushleft]{threeparttable}
\usepackage{textcomp}
\usepackage{subfigure}
\usepackage[toc,page]{appendix}

\newcommand{\Rni}{\mathbb{R}^{n_{i}}}

\newcommand{\nfR}{\mathbb{R}}
\newcommand{\xpi}{x_i}

\newcommand{\xmi}{x_{-i}}
\newcommand{\allx}{\mathbf{x}}

\newcommand{\allu}{\mathbf{u}}
\newcommand{\fpi}{f_{i}}

\newcommand{\lip}{\left<}
\newcommand{\rip}{\right>}
\newcommand{\lvt}{\left[}
\newcommand{\rvt}{\right]}
\newcommand{\idl}{\mbox{Ideal}}
\newcommand{\qmod}{\mbox{Qmod}}
\newcommand{\ddd}{,\ldots,}

\DeclareMathOperator{\Rank}{rank}
\DeclareMathOperator{\dom}{dom}

\newcommand{\re}{\mathbb{R}}
\newcommand{\cpx}{\mathbb{C}}

\newcommand{\N}{\mathbb{N}}

\def\af{\alpha}

\newcommand{\st}{\mathit{s.t.}}
\newcommand{\reff}[1]{(\ref{#1})}

\newcommand{\lmd}{\lambda}

\newcommand{\dt}{\delta}

\newcommand{\mc}[1]{\mathcal{#1}}

\def\rank{\mbox{rank}}

\newcommand{\bdes}{\begin{description}}
\newcommand{\edes}{\end{description}}

\newcommand{\bal}{\begin{align}}
\newcommand{\eal}{\end{align}}

\newcommand{\bnum}{\begin{enumerate}}
\newcommand{\enum}{\end{enumerate}}

\newcommand{\bit}{\begin{itemize}}
\newcommand{\eit}{\end{itemize}}

\newcommand{\bea}{\begin{eqnarray}}
\newcommand{\eea}{\end{eqnarray}}
\newcommand{\be}{\begin{equation}}
\newcommand{\ee}{\end{equation}}

\newcommand{\baray}{\begin{array}}
\newcommand{\earay}{\end{array}}

\newcommand{\bsry}{\begin{subarray}}
\newcommand{\esry}{\end{subarray}}

\newcommand{\bca}{\begin{cases}}
\newcommand{\eca}{\end{cases}}

\newcommand{\bcen}{\begin{center}}
\newcommand{\ecen}{\end{center}}

\newcommand{\bbm}{\begin{bmatrix}}
\newcommand{\ebm}{\end{bmatrix}}

\newcommand{\btab}{\begin{tabular}}
\newcommand{\etab}{\end{tabular}}

\theoremstyle{remark}
\newtheorem*{remark}{Remark}
\theoremstyle{plain}
\newtheorem{thm}{Theorem}[section]{\bf}{\it}
{\bf}{\it}
\newtheorem{prop}[thm]{Proposition}{\bf}{\it}
\newtheorem{lem}[thm]{Lemma}{\bf}{\it}
{\bf}{\it}
{\bf}{\it}
\newtheorem{defi}[thm]{Definition}{\bf}{\it}
\newtheorem{alg}[thm]{Algorithm}{\bf}{\it}
\newtheorem{exm}[thm]{Example}{\bf}{\rm}
{\bf}{\it}
{\bf}{\it}

\numberwithin{equation}{section}

\begin{document}

\title[Convex GNEPs and Polynomial Optimization]
{Convex Generalized Nash Equilibrium Problems and Polynomial Optimization}

\author[Jiawang Nie]{Jiawang~Nie}
\address{Jiawang Nie, Department of Mathematics,
University of California San Diego,
9500 Gilman Drive, La Jolla, CA, USA, 92093.}
\email{njw@math.ucsd.edu}

\author[Xindong Tang]{Xindong~Tang}
\address{Xindong Tang, Department of Applied Mathematics,
The Hong Kong Polytechnic University,
Hung Hom, Kowloon, Hong Kong.}
\email{xindong.tang@polyu.edu.hk}

\date{}

\subjclass[2010]{90C33, 91A10, 90C22, 65K05}
\keywords{Generalized Nash Equilibrium Problem, Convex polynomials, Polynomial optimization, Moment-SOS relaxation,
Lagrange multiplier expression.}

\begin{abstract}
This paper studies convex Generalized Nash Equilibrium Problems (GNEPs)
that are given by polynomials.
We use rational and parametric expressions for Lagrange multipliers
to formulate efficient polynomial optimization for
computing Generalized Nash Equilibria (GNEs).
The Moment-SOS hierarchy of semidefinite relaxations
are used to solve the polynomial optimization.
Under some general assumptions, we prove the method can find a GNE
if there exists one, or detect nonexistence of GNEs.
Numerical experiments are presented to show the efficiency of the method.
\end{abstract}

\maketitle

\section{Introduction}
The Generalized Nash Equilibrium Problem (GNEP) is a kind of game to find strategies
for a group of players such that each player's objective function is optimized,
for given other players' strategies.
Suppose there are  $N$  players and the $i$th player's strategy is a vector
$x_{i}\in\mathbb{R}^{n_{i}}$ (the $n_i$-dimensional real Euclidean space).
We write that
\[x_i:=(x_{i,1},\ldots,x_{i,n_i}),\quad x:=(x_1,\ldots,x_N).\]
The total dimension of all strategies is $n := n_1+ \ldots + n_N.$
The main task of the GNEP is to find a tuple
$u = (u_1, \ldots, u_N)$ of strategies
such that each $u_i$ is a minimizer of the $i$th player's optimization
\be
\label{eq:GNEP}
\mbox{F}_i(u_{-i}): \,
\left\{ \begin{array}{cl}
\min\limits_{\xpi\in \Rni}  &  \fpi(u_1, \ldots, u_{i-1}, x_i, u_{i+1},\ldots, u_N) \\
\st & g_{i,j}(u_1, \ldots, u_{i-1}, x_i, u_{i+1},\ldots, u_N)  = 0 \,
        (j\in \mathcal{E}_i),\\
	& g_{i,j}(u_1, \ldots, u_{i-1}, x_i, u_{i+1},\ldots, u_N)  \geq 0 \,
         (j\in \mathcal{I}_i),
\end{array} \right.
\ee
where $u_{-i} := (u_1, \ldots, u_{i-1}, u_{i+1},\ldots, u_N)$,
the $f_i$ and $g_{i,j}$ are continuously differentiable functions in $x_i$,
and the $\mc{E}_i$, $\mc{I}_i$ are disjoint finite (possibly empty) labeling sets.
The point $u$ satisfying the above is called a Generalized Nash Equilibrium (GNE).
For notational convenience, when the $i$th player's strategy is considered,
we use $x_{-i}$ to denote the subvector of all players' strategies
except the $i$th one, i.e.,
\[
x_{-i} \, := \, (x_1, \ldots, x_{i-1}, x_{i+1}, \ldots, x_N),
\]
and write $x=(x_{i},x_{-i})$ accordingly.

This paper focuses on the Generalized Nash Equilibrium Problem of Polynomials (GNEPP),
i.e., all the functions $f_i$ and $g_{i,j}$ are polynomials in $x$.
For each $i=1,\ldots,N$, let $X_i$ be the point-to-set map such that
\be \label{eq:feaset}
X_i(\xmi) \, := \,
\left\{x_i \in \Rni \left|\begin{array}{l}
g_{i,j}(\xpi,\xmi) = 0 ,\, j\in \mathcal{E}_i,\\
g_{i,j}(\xpi,\xmi) \geq 0 ,\, j\in \mathcal{I}_i
\end{array}\right.\right\}.
\ee
The $X_i(\xmi)$ is the feasible strategy set of $\mbox{F}_i(\xmi)$.
The domain of $X_i$ is
\[\dom(X_i):=\{x_{-i}\in\re^{n-n_i}:X_i(\xmi)\ne\emptyset\}.\]
The tuple $x$ is said to be a feasible point of the GNEP if
$\xpi\in X_i(\xmi)$ for all $i$. Denote the set
\be
X :=
\left\{  x\in\re^n
\left| \begin{array}{l}
 g_{i,j}(\xpi,\xmi) = 0, \, j\in \mathcal{E}_i,\, i=1,\ldots, N,\\
 g_{i,j}(\xpi,\xmi) \geq 0, \, j\in \mathcal{I}_i,\, i=1,\ldots, N
\end{array} \right.
\right \} .
\ee
Then $x$ is a feasible point for the GNEP if and only if $x\in X.$

\begin{defi} \label{df:cgnep} \rm
The GNEP given by (\ref{eq:GNEP}) is called {\it convex}
\footnote{
In some literature, this is also called player-convex,
to distinguish from jointly-convex GNEPs; see \cite{dreves2012nonsmooth}.}
if for all
$i=1,\ldots, N$ and for all given $\xmi\in\dom(X_i)$,
the objective $f_i(\xpi,\xmi)$ is convex in $x_i$ on $X_i(\xmi)$,
all $g_{i,j}(\xpi,\xmi) \, (j\in\mc{E}_i)$
are affine linear in $\xpi$,
and all $g_{i,j}(\xpi,\xmi) \, (j\in\mc{I}_i)$ are concave in $\xpi$.
\end{defi}
For instance, consider the $2$-player GNEPP
\be
\label{eq:firstep}
\begin{array}{lllll}
    \min\limits_{x_{1} \in \re^3 }& \sum\limits_{j=1}^3(x_{1,j}-x_{2,j})^2 &\vline&
    \min\limits_{x_{2} \in \re^3 }& \sum\limits_{j=1}^3\Big((x_{2,j})^4-x_{2,j}\prod\limits_{k=1}^3x_{1,k}\Big)\\
    \st & x_2^Tx_1-1=0,&\vline& \st &\Vert x_1\Vert^2-\Vert x_2\Vert^2\ge0.\\
    & (x_{11}, x_{12}, x_{13}) \ge 0;  				 &\vline&
\end{array}
\ee
In the above, the $\Vert\cdot\Vert$ denotes the Euclidean norm.
For each $i$, the Hessian of $f_i$ with respect to $x_i$
is positive semidefinite for all $\xmi\in\dom(X_i)$.
All players have convex optimization problems, so this is a convex GNEP.
One can directly check that it has a unique GNE $u=(u_1,u_2)$ with
\[
u_1 =\left(\frac{\sqrt[3]{2}}{\sqrt{3}},\frac{\sqrt[3]{2}}{\sqrt{3}},
\frac{\sqrt[3]{2}}{\sqrt{3}}\right), \
u_2 =\left(\frac{1}{\sqrt[6]{108}},\frac{1}{\sqrt[6]{108}},
\frac{1}{\sqrt[6]{108}}\right).
\]

GNEPs originated from economics in
\cite{debreu1952social,arrow1954existence}.
Recently, it has been widely used in many areas, such as economics, transportation, telecommunications and pollution control.
Convex GNEPs often appear in applications.
We refer to \cite{Anselmi2017,ardagna2017generalized,breton2006game,Pang2008}
for recent work on applications of GNEPs.
Some application examples are shown in Section~\ref{sc:ne}.

For the classical Nash Equilibrium Problems (NEPs) of polynomials,
there exist semidefinite relaxation methods~\cite{AZgame,Nie2020nash}.
Convex GNEPs can be reformulated as variational inequality (VI)
or quasi-variational inequality (QVI) problems
\cite{Facchinei2010generalized,harker1991generalized,
nabetani2011parametrized,Pang2005,Han2012}.
The Karush-Kuhn-Tucker (KKT) system for all player's optimization
problems is considered in \cite{dreves2011solution}.
The penalty functions are used to solve convex GNEPs in
\cite{FacKan10,fukushima2011, Facchinei2011partial}.
Some methods using the Nikaido-Isoda function are given in \cite{dreves2012nonsmooth,vonHeusinger2009,vonHeusinger2009-2}.
The Lemke’s method is used to solve affine GNEPs \cite{Schiro2013}.
For general nonconvex GNEPs, we refer to
\cite{Ba2020,dreves2014new,facchinei2009generalized,kanzow2016,Nie2020gs}.
It is generally quite difficult to solve GNEPs, even if they are convex.
This is because the KKT system of a convex GNEP
may still be difficult to solve.
The set of GNEs may be nonconvex,
even for convex NEPs (see \cite{Nie2020nash}).
We refer to \cite{Facchinei2010, Facchinei2010book} for surveys on GNEPs.

\subsection*{Contributions}
This paper focuses on convex GNEPPs.
Under some constraint qualifications,
a feasible point is a GNE if and only if it satisfies the KKT conditions.
We introduce rational and parametric expressions for Lagrange multipliers and
formulate polynomial optimization for computing GNEs.
Our major results are:
\begin{itemize}

\item For GNEPPs, we introduce the rational expression
for Lagrange multipliers and study their properties.
We prove the existence of rational expressions
and give a sufficient and necessary condition for positivity of denominators.
Moreover, we give parametric expressions for Lagrange multipliers for several cases.
For all GNEPs, parametric expressions always exist.

\item Using rational and parametric expressions,
we formulate polynomial optimization
and propose an algorithm for computing GNEs.
Under some general assumptions, we prove that
the algorithm can compute a GNE if it exists, or detect nonexistence of GNEs.
This is the first numerical method that has these properties,
to the best of the authors' knowledge.

\item The Moment-SOS semidefinite relaxations are used
to solve polynomial optimization for finding and verifying GNEs.
Numerical experiments are presented to show the efficiency of the method.

\end{itemize}

The paper is organized as follows.
Some preliminaries about polynomial optimization are given in Section~\ref{sc:pre}.
We introduce rational expressions for Lagrange multipliers in Section~\ref{sc:rtnlmd}.
The parametric expressions for Lagrange multipliers are given in Section~\ref{sc:prt}.
We formulate polynomial optimization problems for computing GNEs and show how to
solve them using the Moment-SOS hierarchy in Section~\ref{sc:alg}.
Numerical experiments and applications are given in Section~\ref{sc:ne}.
Conclusions and some discussions are given in Section~\ref{sc:conc}.

\section{Preliminaries}
\label{sc:pre}

\subsection*{Notation}
The symbol $\mathbb N$ (resp., $\mathbb R$, $\mathbb C$) stands for the set of
nonnegative integers (resp., real numbers, complex numbers).
For a positive integer $k$, denote the set $[k] := \{1, \ldots, k\}$.
For a real number $t$, $\lceil t \rceil$ (resp., $\lfloor t \rfloor$)
denotes the smallest integer not smaller than $t$
(resp., the biggest integer not bigger than $t$).
We use $e_i$ to denote the vector such that the $i$th entry is
$1$ and all others are zeros.
By writing $A\succeq0$ (resp., $A\succ0$), we mean that the matrix $A$
is symmetric positive semidefinite (resp., positive definite).
For the $i$th player's strategy vector $x_i \in \re^{n_i}$,
the $x_{i,j}$ denotes the $j$th entry of $x_i$, for $j = 1, \ldots, n_i$.
When we write $(y,x_{-i})$, it means that the $i$th player's strategy is $y\in\Rni$,
while the vector of all other players' strategy is fixed to be $\xmi$.
Let $\nfR[x]$ denote the ring of polynomials
with real coefficients in $x$,
and $\nfR[x]_d$ denote its subset of polynomials whose degrees are not greater than $d$.
For the $i$th player's strategy vector $x_i$,
the notation $\re[x_i]$ and $\re[x_i]_d$
are defined in the same way.
For $i$th player's objective $f_i(x)$,
the notation $\nabla_{x_i}f_i$, $\nabla^2_{x_i}f_i$
respectively denote its gradient and Hessian with respect to $x_i$.

In the following, we use the letter $z$ to represent either
$x$, $x_i$ or $(x,\omega)$ for some new variables $\omega$,
for convenience of discussion.
Suppose $z := (z_1,\ldots,z_l)$.
For a polynomial $p(z)\in\re[z]$,
the $p=0$ means $p(z)$ is identically zero on $\re^l$.
We say the polynomial $p$ is nonzero if $p\ne0$.
Let $\af := (\af_1, \ldots, \af_l) \in \N^{l}$, and we denote
\[
z^\alpha := z_1^{\alpha_1} \cdots z_l^{\alpha_l}, \quad
|\alpha|:=\alpha_1+\ldots+\alpha_l.
\]
For an integer $d >0$, denote the monomial power set
\[
{\mathbb{N}}_d^l \, := \,
\{\alpha\in {\mathbb{N}}^l: \, \ |\alpha| \le d \}.
\]
We use $[z]_d$ to denote the vector of all monomials in $z$
whose degree is at most $d$, ordered in the graded alphabetical ordering.
For instance, if $z =(z_1, z_2)$, then
\[
[z]_3 = (1,  z_1, z_2, z_1^2, z_1z_2,
z_2^2, z_1^3, z_1^2z_2, z_1z_2^2, z_2^3).
\]
Throughout the paper, a property is said to hold {\it generically}
if it holds for all points in the space of input data
except a set of Lebesgue measure zero.

\subsection{Ideals and positive polynomials}
\label{ssc:poly}

Let $\mathbb{F} := \re\ \mbox{or}\ \cpx$. For a polynomial
$p \in\mathbb{F}[z]$ and subsets $I, J \subseteq \mathbb{F}[z]$,
define the product and Minkowski sum
\[
p \cdot I  :=\{ p  q: \, q \in I \}, \quad
I+J  := \{a+b: \, a \in I, b \in J  \}.
\]
The subset $I$ is an ideal if $p \cdot I \subseteq I$ for all $p\in\mathbb{F}[z]$
and $I+I \subseteq I$.
For a tuple of polynomials $q = (q_1, \ldots, q_m)$, the set
\[
\idl[q]:= q_1\cdot\mathbb{F}[z] + \ldots + q_m \cdot \mathbb{F}[z]
\]
is the ideal generated by $q$, which is the smallest ideal
containing each $q_i$.

We review basic concepts in polynomial optimization.
A polynomial $\sigma \in \re[z]$ is said to be a sum of squares (SOS)
if $\sigma = p_1^2+\ldots+p_k^2$ for some polynomials $p_i \in\nfR[z]$.
The set of all SOS polynomials in $z$ is denoted as $\Sigma[z]$.
For a degree $d$, we denote the truncation
\[
\Sigma[z]_d \, := \, \Sigma[z] \cap \nfR[z]_d.
\]
For a tuple $g=(g_1,\ldots,g_t)$ of polynomials in $z$,
its quadratic module is the set
\[
\qmod[g] \, := \,  \Sigma[z] +  g_1 \cdot \Sigma[z] + \ldots + g_t \cdot  \Sigma[z].
\]
Similarly, we denote the truncation of $\qmod[g]$
\[
\qmod[g]_{2d} \, := \, \Sigma[z]_{2d} + g_1\cdot \Sigma[z]_{2d-\deg(g_1)}
+\ldots+g_t\cdot\Sigma[z]_{2d-\deg(g_t)}.
\]
The tuple $g$ determines the basic closed semi-algebraic set
\begin{equation}
  \label{polyrep}
\mathcal{S}(g) \, := \,  \{z \in \nfR^l: g_1(z) \ge  0, \ldots, g_t(z) \ge  0  \}.
\end{equation}
For a tuple $h=(h_1,\ldots,h_s)$ of polynomials in $\re[z]$, its real zero set is
\[
\mc{Z}(h) := \{z \in\re^l: h_1(z)=\ldots=h_s(z)=0\}.
\]
The set $\idl[h]+\qmod[g]$ is said to be {\it archimedean}
if there exists $\rho \in \idl[h]+\qmod[g]$ such that the set $\mathcal{S}(\rho)$ is compact.
If $\idl[h]+\qmod[g]$ is archimedean, then
$\mathcal{Z}(h)\cap\mathcal{S}(g)$ must be compact.
Conversely, if $\mathcal{Z}(h)\cap\mathcal{S}(g)$ is compact, say,
$\mathcal{Z}(h)\cap\mathcal{S}(g)$ is contained in the ball $R -\|z\|^2 \ge 0$,
then $\idl[h]+\qmod[g,R -\|z\|^2]$ is archimedean
and $\mathcal{Z}(h)\cap\mathcal{S}(g) = \mathcal{Z}(h)\cap\mathcal{S}(g, R -\|z\|^2)$.
Clearly, if $f \in \idl[h]+\qmod[g]$, then
$f \ge 0$ on $\mathcal{Z}(h) \cap \mathcal{S}(g)$.
The reverse is not necessarily true.
However, when $\idl[h]+\qmod[g]$ is archimedean,
if $f > 0$ on $\mathcal{Z}(h)\cap\mathcal{S}(g)$, then $f \in \idl[h]+\qmod[g]$.
This conclusion is referenced as
Putinar's Positivstellensatz \cite{putinar1993positive}.
Interestingly, if $f \ge 0$ on $\mathcal{Z}(h)\cap\mathcal{S}(g)$,
we also have $f\in \idl[h]+\qmod[g]$,
under some standard optimality conditions \cite{nie2014optimality}.

\subsection{Localizing and moment matrices}
\label{ssc:locmat}

Let $\re^{ \N_{2d}^{l} }$ denote the space of all real vectors
that are labeled by $\af \in \N_{2d}^{l}$.
A vector $y \in \re^{ \N_{2d}^{l} }$ is labeled as
\[
y \,=\, (y_\af)_{ \af \in \N_{2d}^{l} }.
\]
Such $y$ is called a
{\it truncated multi-sequence} (tms) of degree $2d$.
For a polynomial $f  = \sum_{ \af \in \N^l_{2d} } f_\af z^\af \in  \re[z]_{2d}$,
define the operation
\be \label{<f,y>}
\langle f, y \rangle \, := \, {\sum}_{ \af \in \N^l_{2d} } f_\af y_\af.
\ee
The operation $\langle f, y \rangle$ is a bilinear function in $(f, y)$.
For a polynomial $q \in \re[z]$, with $\deg(q) \le 2d$,
and the integer $t = d - \lceil \deg(q)/2 \rceil$,
the outer product $q \cdot [z]_t ([z]_t)^T$
is a symmetric matrix polynomial in $z$, with length $\binom{n+t}{t}$.
We write the expansion as
\[
q \cdot [z]_t ([z]_t)^T \, = \, {\sum}_{ \af \in \N_{2d}^l }
z^\af  Q_\af,
\]
for some symmetric matrices $Q_\af$.
Then we define the matrix function
\be \label{df:Lf[y]}
L_{q}^{(d)}[y] \, := \, {\sum}_{ \af \in \N_{2d}^l } y_\af  Q_\af.
\ee
It is called the $d$th {\it localizing matrix} of $q$ generated by $y$.
For given $q$, the matrix $L_{q}^{(d)}[y]$ is linear in $y$.
Localizing and moment matrices are important for getting semidefinite
relaxations of solving polynomial optimization
\cite{Las01,nie2013certifying,nie2013polynomial}.
They are also useful for solving truncated moment problems
\cite{FiaNie12,nie2015linear}
and tensor decompositions \cite{NieGP17,nie2017low}.
We refer to \cite{lasserre2015introduction,LasICM,Lau09,LauICM,NS09,NieLoc}
for more references about polynomial optimization and moment problems.

\subsection{Lagrange multiplier expressions}
\label{ssc:LME}

We study optimality conditions for Generalized Nash Equilibrium Problems.
Consider the $i$th player's optimization.
For convenience, suppose $\mc{E}_i\cup\mc{I}_i = [m_i]$
and $g_i:=(g_{i,1},\ldots,g_{i,m_i})$. For a given $\xmi$, 
under some suitable constraint qualifications
(e.g., the linear independence constraint qualification (LICQ),
Mangasarian-Fromovite constraint qualification (MFCQ),
or the Slater's Condition; see \cite{Brks} for them),
if $x_i$ is a minimizer of $\mbox{F}_i(\xmi)$,
then there exists a Lagrange multiplier vector
$\lambda_i:=(\lambda_{i,1},\ldots,\lambda_{i,m_i})$ such that
\be
\label{eq:KKTwithLM}
\left\{
\begin{array}{l}
  \nabla_{x_i} f_i(x)-\sum_{j=1}^{m_i}\lambda_{i,j}\nabla_{x_i} g_{i,j}(x)=0,\\
  \lambda_i\perp g_i(x),\,g_{i,j}(x)=0\,(j\in\mathcal{E}_i),\\
  \lambda_{i,j}\ge0\,(j\in\mathcal{I}_i),\,g_{i,j}(x)\ge0\,(j\in\mathcal{I}_i).
\end{array}
\right.
\ee
This is called the first order Karush-Kuhn-Tucker system for $\mbox{F}_i(\xmi)$.
Such $(x_i,\lambda_i)$ is called a critical pair of $\mbox{F}_i(\xmi)$.
Therefore, if $x$ is a GNE, under constraint qualifications,
then (\ref{eq:KKTwithLM}) holds for all $i\in[N]$, i.e., there exist
Lagrange multiplier vectors $\lambda_1,\ldots,\lambda_N$ such that
\be
\label{eq:KKTwithLMalp}
\left\{
\begin{array}{l}
\nabla_{x_i} f_i(x)-\sum_{j=1}^{m_i}\lambda_{i,j}\nabla_{x_i} g_{i,j}(x)=0\,(i\in[N]),\\
\lambda_i\perp g_i(x)\,(i\in[N]),\,g_{i,j}(x)=0\,(i\in[N],j\in\mathcal{E}_i),\\
\lambda_{i,j}\ge0\,(i\in[N],j\in\mathcal{I}_i),\,g_{i,j}(x)\ge0\,(i\in[N],j\in\mathcal{I}_i).\\
\end{array}
\right.
\ee
A point $x$ satisfying (\ref{eq:KKTwithLMalp}) is called a KKT point for the GNEP.
For convex GNEPs, each KKT point is a GNE \cite[Theorem~4.6]{Facchinei2010}.

For each critical pair $(x_i,\lambda_i)$ of $\mbox{F}_i(\xmi)$,
the equation (\ref{eq:KKTwithLM}) implies that
\be
\label{eq:Clmd=df}
\underbrace{\bbm
\nabla_{x_i} g_{i,1}(x) & \nabla_{x_i} g_{i,2}(x) &  \cdots &  \nabla_{x_i} g_{i,m_i}(x) \\
g_{i,1}(x) & 0  & \cdots & 0 \\
0  & g_{i,2}(x)  & \cdots & 0 \\
\vdots & \vdots & \ddots & \vdots \\
0  &  0  & \cdots & g_{i,m_i}(x)
\ebm}_{G_i(x) }
\underbrace{
\bbm  \lmd_{i,1} \\ \lmd_{i,2} \\ \vdots \\ \lmd_{i,m_i} \ebm
}_{\lmd_i}
=
\underbrace{\bbm  \nabla_{x_i} f_i(x)  \\ 0 \\ \vdots \\ 0 \ebm}_{ \hat{f}_i(x)} .
\ee
If there exists a matrix polynomial $L_i(x)$ such that
\begin{equation}
\label{eq:LC=I}
L_i(x)G_i(x)=I_{m_i},
\end{equation}
then the Lagrange multipliers $\lambda_i$ can be expressed as
\[\lambda_i=L_i(x)\hat{f}_i(x).\]
The vector of polynomials $\lambda_{i}(x):=(\lambda_{i,1}(x),\ldots,\lambda_{i,m_i}(x))$ is called a \textit{polynomial expression for Lagrange multipliers}\cite{Nie2019},
where $\lambda_{i,j}(x)$ is the $j$th component of $L_i(x)\hat{f}_i(x)$.
The matrix polynomial $G_i(x)$ is said to be nonsingular
if it has full column rank for all $x\in\cpx^n$. It was shown that
$G_i(x)$ is nonsingular if and only if there exists
$L_i(x)\in \re[x]^{(m_i+n_i)\times m_i}$ such that
\reff{eq:LC=I} holds \cite[Proposition~5.1]{Nie2019}.
The nonsingularity of $G_i(x)$
is independent of objective functions or other player's constraints.

For example, consider the GNEP given by (\ref{eq:firstep}).
The first player's optimization has a polynomial expression of Lagrange multipliers
\be
\label{eq:1stlmd}
\lambda_{1,1}= x_1^T\nabla_{x_1}f_1, \, \lambda_{1,j+1}=\frac{\partial{f_1}(x)}{\partial{x_{1,j}}}-
\lambda_{1,1}x_{2,j}\,(j=1,2,3).
\ee
For the second player, the matrix polynomial $G_2(x)$
is not nonsingular, and polynomial expressions do not exist.
In section~\ref{sc:ne}, we give a rational expression 
for the second player's Lagrange multipliers.

\section{Rational expressions for Lagrange Multipliers}
\label{sc:rtnlmd}

In Section~\ref{ssc:LME}, a polynomial expression for
the $i$th player's Lagrange multipliers exists if and only if
the matrix $G_i(x)$ is nonsingular.
For classical NEPs of polynomials, the nonsingularity holds generically \cite{Nie2019,Nie2020nash}.
However, this is often not the case for GNEPs.
Let $g_i=(g_{i,1},\ldots,g_{i,m_i})$ be the tuple of constraining polynomials
in $\mbox{F}_i(x_{-i})$ and $G_i(x)$ be the matrix polynomial as in (\ref{eq:LC=I}).
If there exists a matrix polynomial $\hat{L}_i(x)$
and a nonzero scalar polynomial $q_i(x)$ such that
\be\label{eq:LC=q}
\hat{L}_i(x)G_i(x) \,= \, q_i(x)\cdot I_{m_i} ,
\ee
then $q_i(x)\lambda_{i}=\hat{L}_i(x)\hat{f}_i(x)$
for all critical pairs $(x_i,\lambda_i)$ of $\mbox{F}_i(\xmi)$. Let
\be
\label{eq:rtnlmd_i}
\hat{\lambda}_{i}(x) \,:= \,
\hat{L}_i(x)\hat{f}_i(x).
\ee
Denote by $\hat{\lambda}_{i,j}(x)$ the $j$th entry of $\hat{\lambda}_{i}(x)$.

\begin{defi} \label{LME:rat} \rm
For the $i$th player's optimization $\mbox{F}_i(\xmi)$,
if there exist polynomials $\hat{\lambda}_{i,1},\ldots,\hat{\lambda}_{i,m_i}$ 
and a nonzero polynomial $q_i$ such that
$q_i(x)\ge 0$ for all $x\in X$, and $\hat{\lambda}_{i,j}(x)=q_i(x)\lambda_{i,j}$ 
holds for all critical pairs $(x_i,\lambda_i)$,
then we call the tuple
\[
\hat{\lambda}_i/q_i:=(\hat{\lambda}_{i,1}(x)/q_i(x),
\ldots,\hat{\lambda}_{i,m_i}(x)/q_i(x))
\]
a rational expression for Lagrange multipliers.
\end{defi}

The following is an example of rational expression.

\begin{exm}\rm
\label{ep:>0}
Consider the 2-player convex GNEP
\be
\begin{array}{lllll}
\label{eq:rational_example}
\min\limits_{x_1\in\re^2} & f_1(x_1,x_2) & \vline & \min\limits_{x_2\in\re^1} & f_2(x_1,x_2)\\
\st              & 2-x_1^Tx_1-x_2\ge0; & \vline & \st              & 3x_2-x_1^Tx_1\ge0 ,\,1-x_2\ge0.
\end{array}
\ee	
The matrices of polynomials $G_1(x)$ and $G_2(x)$ are
\[
G_1(x):=\lvt
\begin{array}{c}
-2x_{1,1}\\
-2x_{1,2}\\
2-x_1^Tx_1-x_2
\end{array}\rvt,\quad
G_2(x):=\lvt
\begin{array}{cc}
3 & -1\\
3x_2-x_1^Tx_1 & 0 \\
0 & 1-x_2
\end{array}\rvt.\]
For $x_1=(0,0)$ and $x_2=2$, the $G_1(x)$ is the zero vector.
For $x_1=(\sqrt{3},0)$ and $x_2=1$, $\rank(G_2(x))=1$.
Both $G_1(x),G_2(x)$ are not nonsingular,
so there are no polynomial expressions for Lagrange multipliers.
However, (\ref{eq:LC=q}) holds for
\be  \label{eq:L1L2>0}
\baray{rl}
q_1(x)=2-x_2,   & q_2(x)=1-\frac{1}{3}x_1^Tx_1, \\
\hat{L}_1(x)=\lvt\begin{array}{ccc}
-\frac{x_{1,1}}{2} & -\frac{x_{1,2}}{2} & 1
\end{array}\rvt,
  &
\hat{L}_2(x)=\lvt\begin{array}{ccc}
\frac{1}{3}-\frac{1}{3}x_2 & \frac{1}{3} & \frac{1}{3}\\
\frac{1}{3}x_1^Tx_1-x_2 & 1 & 1
\end{array}\rvt .
\earay
\ee
The Lagrange multiplier expressions are
\be
\label{eq:lmd_>0}
\lambda_1=\frac{-x_1^T\nabla_{x_1}f_1}{2q_1},\,
\lambda_{2,1}=\frac{(1-x_2)}{3q_2}\cdot\frac{\partial{f_2}}{\partial{x_2}},\,
\lambda_{2,2}=\frac{x_1^Tx_1-3x_2}{3q_2}\cdot\frac{\partial{f_2}}{\partial{x_2}}.
\ee
\end{exm}

In section~\ref{sc:nonnegalmd},
we show that if none of the $g_{i,j}$ is identically zero,
then a rational expression for $\lambda_i$ always exists.

\subsection{Optimality conditions and rational expressions}

Suppose for each $i$, there exists a rational expression
$\hat{\lambda}_i/q_i$ for the $i$th player's Lagrange multiplier vector.
Since $q_{i}(x)\lambda_{i,j}=\hat{\lambda}_i(x)$
and $q_i(x)\ge0$ for all $x\in X$,
the following holds for all KKT points
\be
\label{eq:KKTrational}
\left\{
\begin{array}{l}
q_i(x)\nabla_{x_i} f_i(x)-\sum\nolimits_{j=1}^{m_i}
 \hat{\lambda}_{i,j}(x)\nabla_{x_i} g_{i,j}(x)=0 \, (i\in[N]),\\
\hat{\lambda}_i(x)\perp g_i(x),g_{i,j}(x)=0 \, (j\in\mathcal{E}_i,i\in[N]),\\
g_{i,j}(x)\ge0,\hat{\lambda}_{i,j}(x)\ge0\,(j\in\mathcal{I}_i,i\in[N]).
\end{array}
\right.
\ee
Under some constraint qualifications,
if $x$ is a GNE, then it satisfies (\ref{eq:KKTrational}).
For convex GNEPs, if $x$ satisfies (\ref{eq:KKTrational}) and $q_i(x)>0$,
then $x$ must be a GNE, since it satisfies
(\ref{eq:KKTwithLMalp}) with $\lambda_{i,j}$ given by
$\lambda_{i,j}=\hat{\lambda}_{i,j}(x)/{q_i(x)}.$
This leads us to consider the following optimization problem
\be
\label{eq:KKTfeasoptrtn}
\left\{
\begin{array}{rll}
\min\limits_{x\in X} & [x]_1^T\Theta[x]_1 \\
\st &q_i(x)\nabla_{x_i} f_i(x)-\sum\nolimits_{j=1}^{m_i}\hat{\lambda}_{i,j}(x)\nabla_{x_i} g_{i,j}(x)=0\,(i\in[N]),\\
	&\hat{\lambda}_{i,j}(x)\perp g_{i,j}(x)\,(j\in\mathcal{E}_i\cup\mc{I}_i,i\in[N]),\\
	&\hat{\lambda}_{i,j}(x)\ge0\,(j\in\mathcal{I}_i,i\in[N]).
\end{array}
\right.
\ee
In the above,
$\Theta$ is a generically chosen positive definite matrix.
The following proposition is straightforward.
\begin{prop}
\label{pr:rtn}
For the GNEPP given by (\ref{eq:GNEP}), suppose for each $i\in[N]$,
the Lagrange multiplier vector $\lambda_i$
has the rational expression as in Definition~\ref{LME:rat}.
\begin{enumerate}[(i)]
\item If (\ref{eq:KKTfeasoptrtn}) is infeasible,
then the GNEP has no KKT points. Therefore,
if every GNE is a KKT point,
then the infeasibility of (\ref{eq:KKTfeasoptrtn})
implies the nonexistence of GNEs.

\item Assume the GNEP is convex.
If $u$ is a feasible point of (\ref{eq:KKTfeasoptrtn})
and $q_i(u)>0$ for all $i\in[N]$, then $u$ must be a GNE.
\end{enumerate}
\end{prop}

In Proposition~\ref{pr:rtn} (ii), if $q_i(u)=0$,
then $u$ may not be a GNE. The following is such an example.

\begin{exm}\rm
\cite[Example A.8]{FacKan10}
\label{ep:A8}
Consider the 3-player convex GNEP
\[
\begin{array}{llllllll}
\min\limits_{x_1\in\re^1} & -x_1   & \vline & \min\limits_{x_2\in\re^1} & (x_2-0.5)^2 & \vline & \min\limits_{x_3\in\re^1} & (x_3-1.5x_1)^2\\
\st     & x_3\le x_1+x_2\le 1, & \vline & \st              & x_3\le x_1+x_2\le 1, & \vline & \st & 0\le x_3\le2.\\
        & x_1\ge0; & \vline &  & x_2\ge0; & \vline &
\end{array}
\]
For the first two players ($i=1,2$), the equation (\ref{eq:LC=q}) holds for
\[
\hat{L}_i(x) := \lvt
\begin{array}{cccc}
x_i(1-x_1-x_2) & x_i & x_i & x_1+x_2-1\\
x_i(x_3-x_1-x_2) & x_i & x_i & x_1+x_2-x_3\\
0 & 0 & 0 & 1-x_3
\end{array}\rvt,\ q_i(x) := x_i(1-x_3).
\]
For the third player ($i=3$), the equation (\ref{eq:LC=q}) holds for
\[
\hat{L}_3(x) := \frac{1}{2}\cdot\lvt
\begin{array}{cccc}
2-x_3 & 1 & 1\\
-x_3 & 1 & 1
\end{array}\rvt,
\quad q_3 := 1.
\]
The Lagrange multiplier expressions can be obtained by letting
$\hat{\lambda}_i(x):=\hat{L}_i(x)\hat{f}_i(x)$.
It is clear that
$u_1=0,u_2=0.5,u_3=0$ satisfy (\ref{eq:KKTwithLMalp}) with $q_1(u) =0$.
However, $u_1=0$ is not a minimizer for the first player's optimization
$\mbox{F}_1(u_{-1})$. It is interesting to note that for
$u_1 =\frac{2}{3}, u_2 =\frac{1}{3}, u_3 =1$,
the tuple $u = (u_1, u_2, u_3)$ satisfies (\ref{eq:KKTwithLMalp})
with $q_1(u)=q_2(u)=0$, but $u$ is still a GNE \cite{FacKan10}.
\end{exm}

We would like to remark that for some special GNEPs,
the equality $q_i(u)=0$ may imply that $u_i$ is a minimizer of $\mbox{F}_i(u_{-i})$.
See Example~\ref{ep:jointball} for such a case.

\subsection{Existence of rational expressions}
\label{sc:nonnegalmd}

We study the existence of rational expressions with nonnegative $q_i(x)$.
The following is a useful lemma.

\begin{lem}
\label{lm:rtnext}
For the $i$th player's optimization $\mbox{F}_i(\xmi)$,
if every $g_{i,j}(x)$ is not identically zero,
then a rational expression exists for $\lambda_i$.
\end{lem}
\begin{proof}
Let $H_i(x)=G_i(x)^TG_i(x)$, where $G_i(x)$ 
is the matrix polynomial in \reff{eq:Clmd=df}.
If every $g_{i,j}(x)$ is not identically zero,
then the determinant $\det H_i(x)$ is also not identically zero.
Let $\mbox{adj}\,H_i(x)$ denote the adjoint matrix of $H_i(x)$, then
\[ H_i(x)  \cdot \mbox{adj}\, H_i(x)  \, = \, \det H_i(x) \cdot I_{m_i}.\]
For $\hat{L}_i(x) :=\mbox{adj}\,H_i(x)\cdot G_i(x)^T$,
we get the rational expression
\be
\label{eq:trivialL}
\lambda_{i,j}(x) \,= \,
\frac{1}{ \det H_i(x)}\hat{L}_i(x)\cdot\hat{f}_i(x) .
\ee
Moreover, $q_i(x) \geq 0$ for all $x$, since $H_i(x)$
is positive semidefinite everywhere.
\end{proof}

The rational expression in (\ref{eq:trivialL}) may not be very practical,
because the determinantal polynomials often have high degrees.
In practice, we usually have rational expressions with low degrees.
If each $q_i(x)>0$ for all $x\in X$, then every solution of (\ref{eq:KKTfeasoptrtn})
is a GNE. One wonders when a rational expression exists
with $q_i(x)>0$ on $X$.
The matrix polynomial $G_i$ is said to be \textit{nonsingular on $X$}
if $G_i(x)$ has full column rank for all $x\in X$.
For the GNEP given in Example~\ref{ep:>0}, both $G_1(x)$ and $G_2(x)$
are nonsingular on $X$.
The following proposition is useful.

\begin{prop} \label{pr:nonsingoverK}
The matrix $G_i(x)$ is nonsingular on $X$ if and only if
there exists a matrix polynomial $\hat{L}_i(x)$
satisfying \reff{eq:LC=q} with $q_i(x)>0$ on $X$.
\end{prop}
\begin{proof}
First, if the matrix polynomial $G_i(x)$ has full column rank for all $x\in X$,
let $H_i(x) :=G_i(x)^TG_i(x)$,
then $H_i(x)$ is positive definite and the determinant $\det H_i(x)>0$
for all $x \in X$. Therefore, for
$\hat{L}_i(x) := \mbox{adj}\, H_i(x)$,
the equation (\ref{eq:trivialL}) is satisfied with
$q_i(x) := \det H_i(x)>0$ over $X$.
Second, if \reff{eq:LC=q} holds with $q_i(x)>0$ on $X$,
then $G_i(x)$ is clearly nonsingular on $X$.
\end{proof}

\begin{remark}
If $G_i(x)$ is nonsingular on $X$,
then the LICQC must hold for the $i$th player's optimization.
Furthermore, if this holds for all $i\in[N]$,
then all GNEs are KKT points.
\end{remark}

\subsection{A numerical method for finding rational expressions}

We give a numerical method for finding rational expressions for Lagrange multipliers.
It was introduced in \cite{Nie2020bilevel} for solving bilevel optimization problems.
Let $G_i(x)$ be the matrix polynomial defined in (\ref{eq:Clmd=df}).
For convenience, denote the tuples
\[
g_{\mc{E}}:= (g_{i,j})_{ i\in[N],j\in\mc{E}_i},\quad
g_{\mc{I}}:= (g_{i,j})_{ i\in[N],j\in\mc{I}_i }.
\]
For a priori degree $d$,
consider the following linear convex optimization:
\be  \label{eq:findL}
\left \{ \begin{array}{rl}
\max\limits_{\hat{L}_i,q_i,\gamma}  & \gamma  \\
\st & \hat{L}_i\cdot G_i= q_i\cdot I_{m_i},\, q_i( v )=1, \\
	& q_i-\gamma\in\idl[g_{\mc{E}}]_{2d}+\qmod[g_{\mc{I}}]_{2d}, \\
	& \hat{L}_i\in(\re[x]_{2d-\deg{G_i}})^{m_i\times (m_i+n_i)}.
\end{array} \right.
\ee
In the above, the first equality is the same as (\ref{eq:LC=q}).
The second equality ensures that $q_i$ is not identically zero,
where $v$ is a priori point in $X$.
The constraint $q_i-\gamma\in\idl[g_{\mc{E}_i}]+\qmod[g_{\mc{E}_i}]$
forces the $q_i(x) \ge \gamma$ on $X$.
Therefore, if the maximum $\gamma$ is positive, then
$q_i(x)>0$ on $X$. By Lemma~\ref{lm:rtnext}, one can always find
a feasible $\gamma \ge 0$ satisfying (\ref{eq:findL}), for some $d\le\deg (H(x))$,
if none of $g_{i,j}(x)$ is identically zero.
By Proposition~\ref{pr:nonsingoverK}, if each $G_i(x)$
is nonsingular on $X$ and the archimedeanness holds for $X$,
then there must exist $\gamma > 0$ satisfying (\ref{eq:findL}) for some $d$.
If $(\hat{L}_i,q_i,\gamma)$ is a feasible point of \reff{eq:findL},
then one can get a rational expression for Lagrange
multipliers by letting $\hat{\lambda}_{i,j}(x)=\hat{L}_i(x)\hat{f}_i(x)$.

\begin{exm}\rm
Consider the GNEP in Example~\ref{ep:>0}. We have
\[g_{\mc{E}}=\emptyset,\ g_{\mc{I}}=(2-x_1^Tx_1-x_2,3x_2-x_1^Tx_1,1-x_2).\]
Let $\hat{L}_1(x)$ and $\hat{L}_2(x)$ be the matrix polynomials in (\ref{eq:L1L2>0}),
and $q_1(x)=2-x_2,q_2(x)=1-\frac{1}{3}x_1^Tx_1$.
Let $ v :=(0,0,1)$ for both players,
and $\gamma_1=1$, $\gamma_2=1/2$.
Then, the $(\hat{L}_i(x),q_i(x),\gamma_i)$
is a feasible point of (3.11), for each $i=1,2$.
In fact, we have
\[
\begin{array}{l}
q_1( v )=q_2( v )=1, \quad
q_1(x)-\gamma_1=1-x_2=0+1\cdot(1-x_2)\in\qmod[g_{\mc{I}}]_{2},\\
q_2(x)-\gamma_2=\frac{1}{2}-\frac{1}{3}x_1^Tx_1=\frac{1}{4}(2-x_1^Tx_1-x_2)
+\frac{1}{12}(3x_2-x_1^Tx_1)\in\qmod[g_{\mc{I}}]_{2}.
\end{array}
\]
The rational expressions for Lagrange multipliers are given by \reff{eq:lmd_>0}.
\end{exm}

\begin{exm}\rm
\label{ep:jointball}
Consider the following GNEP
\[
\begin{array}{lllll}
\min\limits_{x_1\in\re^3} & f_1(x_1,x_2) & \vline & \min\limits_{x_2\in\re^3} & f_2(x_1,x_2)\\
\st              & 1-x_1^Tx_1-x_2^Tx_2\ge0; & \vline & \st              & 1-x_1^Tx_1-x_2^Tx_2\ge0.
\end{array}
\]
The constraining tuples $g_{\mc{E}}:=\emptyset,\ g_{\mc{I}}:=(1-x_1^Tx_1-x_2^Tx_2).$
Let $ v :=(0,0,0)$, $\gamma_1=\gamma_2=0$, $q_1(x)=1-x_2^Tx_2$, $q_2(x)=1-x_1^Tx_1$, and
\[\hat{L}_1=\lvt-\frac{1}{2}x_{1,1},\,-\frac{1}{2}x_{1,2},\,-\frac{1}{2}x_{1,3},\,1\rvt,\
\hat{L}_2=\lvt-\frac{1}{2}x_{2,1},\,-\frac{1}{2}x_{2,2},\,-\frac{1}{2}x_{2,3},\,1\rvt.\]
One can verify that $q_1( v )=q_2( v )=1$ and
\[
\begin{array}{l}
q_1(x)-\gamma_1=1-x_2^Tx_2=x_1^Tx_1+1\cdot (1-x_1^Tx_1-x_2^Tx_2)\in\qmod[g_{\mc{I}}]_{2},\\
q_2(x)-\gamma_2=1-x_1^Tx_1=x_2^Tx_2+1\cdot (1-x_1^Tx_1-x_2^Tx_2)\in\qmod[g_{\mc{I}}]_{2}.
\end{array}
\]
By Proposition~\ref{pr:nonsingoverK},
we know $(\hat{L}_1(x),q_1(x),\gamma_1)$ and $(\hat{L}_2(x),q_2(x),\gamma_2)$ are minimizers of (\ref{eq:findL}) for $i=1,2$ respectively.
Therefore, we get the rational expression
\be\label{eq:lmd_ball}
\lambda_1=\frac{-x_1^T \nabla_{x_1}f_1}{2\cdot q_1(x)}, \
\lambda_2=\frac{-x_2^T \nabla_{x_2}f_2}{2\cdot q_2(x)}.
\ee
For each $i=1,2$, if $q_i(x)=0$,
then $0\leq {x_i}^Tx_i\leq 1-{x_{-i}}^Tx_{-i}=0.$
This implies $x_i=(0,0,0)$ is the only feasible point of the $i$th player's optimization and hence it is the minimizer.
Therefore, each feasible point of (\ref{eq:KKTfeasoptrtn}) is a GNE.
\end{exm}

One can solve (\ref{eq:findL}) numerically for getting rational expressions.
This is done in Example~\ref{ep:num_lmd}.

\section{Parametric expressions for Lagrange multipliers}
\label{sc:prt}

For some GNEPs, it may be difficult to find convenient rational expressions
for Lagrange multipliers. Sometimes, the denominators may have high degrees.
This is the case especially when $m_i>n_i$.
If some $q_i$ has high degree, the polynomial optimization
(\ref{eq:KKTrational}) also has a high degree,
which makes the result moment SDP relaxations (see subsections~\ref{sc:opt4all} and \ref{sc:opt4one}) very difficult to be solved.
To fix such issues, we introduce parametric expressions for Lagrange multipliers.

\begin{defi} \label{def:para:LME} \rm
For the $i$th player's optimization $\mbox{F}_i(\xmi)$,
a parametric expression for the Lagrange multipliers is a tuple of polynomials
\[ \hat{\lambda}_{i}(x, \omega_i) := (\hat{\lambda}_{i,1}(x, \omega_i),\ldots,\hat{\lambda}_{i,m_i}(x, \omega_i) ), \]
in $x$ and in a parameter $\omega_i:=(\omega_{i,1},\ldots,\omega_{i,s_i})$ with $s_i\le m_i$,
such that $(x_i,\lambda_i)$ is a critical pair if and only if
there is a value of $\omega_i$ such that (\ref{eq:KKTwithLM}) is satisfied for
$\lambda_{i,j}=\hat{\lambda}_{i,j}(x,\omega_i)$ with $j\in[m_i]$.
\end{defi}

The following is an example of parametric expressions.

\begin{exm}\rm
\label{ep:partialep}
Consider the 2-player convex GNEP
\[
\begin{array}{lllll}
\min\limits_{x_1\in\re^2} & f_1(x_1,x_2) & \vline &
           \min\limits_{x_2\in\re^2} & f_2(x_1,x_2)   \\
\st              & x_{1,1}-2x_{1,2}+x_{2,2}\ge0, & \vline
              & \st &x_{1,2}+x_{2,2}-x_{2,1}^2+1\ge0, \\
 & 1-x_{2,1}\cdot x_1^Tx_1\ge0,& \vline&&2-x_{2,2}\ge0,1+x_{2,2}\ge0,\\
				 & x_{1,1}\ge0,x_{1,2}\ge0;&   \vline&&x_{2,1}\ge0.\\
\end{array}
\] 
The Lagrange multipliers can be expressed as
\be
\label{eq:lmd_partep}
\left\{
\begin{array}{l}
\lambda_{1,1}= \omega_{1,1},\\
\lambda_{1,2}=\frac{1}{2}x_{1,1}(\frac{\partial{f_1}}{\partial{x_{1,1}}}-\omega_{1,1})+\frac{1}{2}x_{1,2}(\frac{\partial{f_1}}{\partial{x_{1,2}}}+2\omega_{1,1}),\\
\lambda_{1,3}=\frac{\partial{f_1}}{\partial{x_{1,1}}}-\omega_{1,1}+2x_{2,1}x_{1,1}\lambda_{1,2},\\
\lambda_{1,4}=\frac{\partial{f_1}}{\partial{x_{1,2}}}+2\omega_{1,1}+2x_{2,1}x_{1,2}\lambda_{1,2};\\
\lambda_{2,1}=\omega_{2,1},\\
\lambda_{2,2}=-\frac{1}{3}\cdot\left[(\frac{\partial{f_2}}{\partial{x_{2,1}}}+2x_{2,1}\omega_{2,1})x_{2,1}+(\frac{\partial{f_2}}{\partial{x_{2,2}}}-\omega_{2,1})(x_{2,2}+1)\right],\\
\lambda_{2,3}=\frac{\partial{f_2}}{\partial{x_{2,2}}}+\lambda_{2,2}-\omega_{2,1},\\
\lambda_{2,4}=\frac{\partial{f_2}}{\partial{x_{2,1}}}+2x_{2,1}\omega_{2,1}.
\end{array}\right.
\ee
\end{exm}

Parametric expressions are quite useful for solving the GNEPs.
The following are some useful cases.
\begin{itemize}

\item [(i)] Suppose the $i$th player's optimization $\mbox{F}_i(\xmi)$
contains the nonnegative constraints, i.e., its constraints are
\[
x_{i,1} \ge 0,\ldots, x_{i,n_i} \ge 0, \quad
g_{i,j}(x) \ge 0 \, (j=n_i+1, \ldots, m_i).
\] 
Let $s_i := m_i-n_i$, then a parametric expression is
\be
\label{eq:nonnegalmd}
\boxed{
\begin{array}{l}
(\lambda_{i,1},\ldots,\lambda_{i,n_i})=
\nabla_{x_i}f_i-\sum_{k=1}^{s_i}\omega_{i,k}\cdot\nabla_{x_i}g_{i,k+n_i},\\
 (\lambda_{i,n_i+1},\ldots,\lambda_{i,m_i}) =(\omega_{i,1},\ldots,\omega_{i,s_i}) .
\end{array}
}
\ee

\item [(ii)]  Suppose the $i$th player's optimization $\mbox{F}_i(\xmi)$
contains box constraints, i.e., its constraints are
\[
\baray{rl}
x_{i,j}-a_{i,j} \ge 0, \, b_{i,j}-x_{i,j} \ge 0, & j = 1,\ldots, n_i \\
g_{i,j}(x) \ge 0. & j=n_i+1, \ldots, m_i
\earay
\]
Let $s_i :=m_i-2n_i$, then a parametric expression is
\be
\label{eq:boxlmd}
\boxed{
\begin{array}{ll}
\lambda_{i,j}= \frac{b-x_{i,j}}{b-a}\cdot\left(\frac{\partial{f_i}}{\partial{x_{i,j}}}-
\sum_{k=1}^{s_i}\omega_{i,k}\cdot\frac{\partial{g_{i,k+2n_i}}}{\partial{x_{i,j}}}\right), & j=1,3,\ldots,2n_i-1 \\
\lambda_{i,j}= \frac{a-x_{i,j}}{b-a}\cdot\left(\frac{\partial{f_i}}{\partial{x_{i,j}}}-
\sum_{k=1}^{s_i}\omega_{i,k}\cdot\frac{\partial{g_{i,k+2n_i}}}{\partial{x_{i,j}}}\right), & j=2,4,\ldots,2n_i \\
\lambda_{i,j}= \omega_{i,j-2n_i}.	
 & j=2n_i+1,\ldots,m_i
\end{array}
}
\ee

\item [(iii)] Suppose the $i$th player's optimization $\mbox{F}_i(\xmi)$
contains simplex constraints, i.e., its constraints are
\[
1-e^Tx_i \ge 0, x_{i,1} \ge 0, \ldots, x_{i,n_i} \ge 0, \,
g_{i,j}(x) \ge 0, \, j=n_i+2,\ldots, m_i.
\]
Let $s_i := m_i-n_i-1$, then a parametric expression is
\be \label{eq:simplexlmd}
\boxed{
\begin{array}{ll}
\lambda_{i,j}=(\nabla_{x_i}f_i-\sum_{k=1}^{s_i}\omega_{i,k}\cdot\nabla_{x_i}g_{i,k+n_i+1})^Tx_i, & j=1\\
\lambda_{i,j}= \frac{\partial{f_i}}{\partial{x_{i,j-1}}}-
\sum_{k=1}^{s_i}\omega_{i,k}\cdot\frac{\partial{g_{i,k+n_i+1}}}{\partial{x_{i,j-1}}}-
\lambda_{i,1}, & j=2,\ldots,n_i+1\\
\lambda_{i,j}= \omega_{i,j-n_i-1}. & j=n_i+2,\ldots,m_i
\end{array}
}
\ee

\item [(iv)] Suppose the $i$th player's optimization $\mbox{F}_i(\xmi)$
contains linear constraints, i.e., its constraints are
\[
a_j^Tx_i - b_j(x_{-i}) \ge 0, \, j=1,\ldots, r, \quad
g_{i,j}(x) \ge 0, \, j=r+1,\ldots, m_i,
\]
where each $b_j$ is a polynomial in $x_{-i}$.
Let $A=\bbm a_1 & \cdots & a_r \ebm^T$. Assume $\rank{A}=r$.
If we let $s_i := m_i-r$, then a parametric expression is
\[
\boxed{
\begin{array}{l}
(\lambda_{i,1},\ldots,\lambda_{i,r})=(AA^T)^{-1}A(
\nabla_{\xpi}f_i-\sum_{k=1}^{s_i}\omega_{i,k}\cdot\nabla_{x_i}g_{i,k+r}),\\ (\lambda_{i,r+1},\ldots,\lambda_{i,m_i}) =
(\omega_{i,1},\ldots,\omega_{i,s_i}).
\end{array}
}
\]

\item [(v)]
Suppose there exists a labeling subset
$\mc{T}_i :=(t_1,\ldots,t_r)\subseteq [m_i]$ such that
\[
\label{eq:lchat}
\hat{G}_{i}(x):=\left[
\begin{array}{rrr}
\nabla_{x_i} g_{i,t_1}(x) & \ldots & \nabla_{x_i} g_{i,t_r}(x)\\
 g_{i,t_1}(x) &&\\
 & \ddots &\\
 &&g_{i,t_r}(x)
\end{array}
\right]
\]
is nonsingular for all $x \in \cpx^n$.
By \cite[Proposition~5.1]{Nie2019}, there exists a matrix polynomial $D_{i}(x)$ such that $D_{i}(x)\cdot \hat{G}_{i}(x)=I_r$.
Let $s_i := m_i-r$, then a parametric expression is
\[
\boxed{
\begin{array}{l}
(\lambda_{i,1},\ldots,\lambda_{i,r})=D_{i}(x)(\nabla_{x_i}f_i-
\sum_{k=1}^{s_i}\omega_{i,k}\cdot\nabla_{x_i}g_{i,k+r}),\\
 (\lambda_{i,r+1},\ldots,\lambda_{i,m_i})= (\omega_{i,1},\ldots,\omega_{i,s_i}).
\end{array}
}
\]

\end{itemize}

We would like to remark that parametric expressions
for Lagrange multipliers always exist.
For instance, one can get a parametric expression by letting
$\omega_{i,j} = \lambda_{i,j}$ for all $j$.
Such expression is called a {\it trivial parametric expression}.
However, it is preferable to have small $s_i$, to save computational costs.

\subsection{Optimality conditions and parametric expressions}

Suppose all players have parametric expressions for their Lagrange multipliers
as in Definition~\ref{def:para:LME}. 
Let $s := s_1+\ldots+s_N$, and denote
\[
\allx \,:=\, (x,\omega_1,\ldots,\omega_N). 
\]
The optimality conditions (\ref{eq:KKTwithLMalp})
can be equivalently expressed as
\be
\label{eq:KKTpartial}
\left\{
\begin{array}{l}
\nabla_{x_i} f_i(x)-\sum_{j=1}^{m_i}\hat{\lambda}_{i,j}(\allx)\nabla_{x_i} g_{i,j}(x)=0 \, (i\in[N]),\\
\hat{\lambda}_i(\allx)\perp g_i(x),g_{i,j}(x)=0 \, (j\in\mathcal{E}_i, i\in[N]),\\
g_{i,j}(x)\ge0,\hat{\lambda}_{i,j}(\allx)\ge0\,(j\in\mathcal{I}_i,i\in[N]).
\end{array}
\right.
\ee
For convex GNEPs, a point $x$ is a GNE if and only if there exists
$\omega:= (\omega_1,\ldots,\omega_N)$ such that $\allx$ satisfies (\ref{eq:KKTpartial}).
Therefore, we consider the optimization
\be
\label{eq:KKTfeasoptprt}
\left\{
\begin{array}{rll}
\min\limits_{\allx\in X\times\re^s} & \lvt\allx\rvt_1^T\Theta\lvt\allx\rvt_1 \\
\st &\nabla_{x_i} f_i(x)-\sum\nolimits_{j=1}^{m_i}\hat{\lambda}_{i,j}(\allx)\nabla_{x_i} g_{i,j}(x)=0 \, (i\in[N]),\\
	&\hat{\lambda}_{i,j}(\allx)\perp g_{i,j}(x)\,(j\in\mathcal{E}_i\cup\mc{I}_i,i\in[N]),\\
	&\hat{\lambda}_{i,j}(\allx)\ge0\,(j\in\mathcal{I}_i,i\in[N]).
\end{array}
\right.
\ee
In the above, the $\Theta$ is a generically chosen positive definite matrix.
The following proposition is straightforward.
\begin{prop}
\label{pr:prt}
For the GNEPP given by (\ref{eq:GNEP}),
suppose each player's optimization has a
parametric expression for their Lagrange multipliers
as in Definition~\ref{def:para:LME}.
\begin{enumerate}[(i)]

\item If (\ref{eq:KKTfeasoptprt}) is infeasible,
then the GNEP has no KKT points. If every GNE is a KKT point,
then the infeasibility of (\ref{eq:KKTfeasoptprt})
implies nonexistence of GNEs.

\item Assume the GNEP is convex.
If $(u,w)$ is a feasible point of (\ref{eq:KKTfeasoptprt}),
then $u$ is a GNE.
\end{enumerate}
\end{prop}

\section{The polynomial optimization reformulation}
\label{sc:alg}

In this section, we give an algorithm for solving convex GNEPs.
We assume each $\lambda_i$ has either a rational or parametric expression,
as in Definition~\ref{LME:rat} or \ref{def:para:LME}.
If $\lambda_i$ has a polynomial or parametric expression,
we let $q_i(x) :=1$.
If $\lambda_i$ has a polynomial or rational expression,
then we let $s_i=0$. Recall the notation
\[
\allx \,:= \, (x,\omega_1,\ldots,\omega_N).
\]
Choose a generic positive definite matrix $\Theta$.
Then solve the following polynomial optimization
\be
\label{eq:KKTfeasopt}
\left\{
\begin{array}{rll}
\min\limits_{\allx} & \lvt\allx\rvt_1^T\Theta\lvt\allx\rvt_1 \\
\st &q_i(x)\nabla_{x_i} f_i(x)-\sum\nolimits_{j=1}^{m_i}\hat{\lambda}_{i,j}(\allx)\nabla_{x_i} g_{i,j}(x)=0 \, (i\in[N]),\\
	&\hat{\lambda}_{i,j}(\allx)\perp g_{i,j}(x)\,(j\in\mathcal{E}_i\cup\mc{I}_i,i\in[N]),\\
	&g_{i,j}(x)=0\,(j\in\mathcal{E}_i,i\in[N]),\\
	&g_{i,j}(x)\ge0\,(j\in\mathcal{I}_i,i\in[N]),\\
	&\hat{\lambda}_{i,j}(\allx)\ge0\,(j\in\mathcal{I}_i,i\in[N]).
\end{array}
\right.
\ee
If (\ref{eq:KKTfeasopt}) is infeasible,
then there are no KKT points.
Since $\Theta$ is positive definite, if (\ref{eq:KKTfeasopt}) is feasible,
then it must have a minimizer, say, $(u,w)\in X\times \re^s$.
For convex GNEPs, if $q_i(u)>0$ for all $i$, then $u$ must be a GNE.
If $q_i(u) \le 0$ for some $i$,
then $u$ may or may not be a GNE. To check this,
we solve the following optimization problem
for those $i$ with $q_i(u) \le 0$
\begin{equation}
\left\{
\begin{array}{llll}
\delta_i := & \min\limits_{\xpi}& f_i(x_i,u_{-i})-f_i(u_{i},u_{-i})&\\
 &\st  &  g_{i,j}(x_{i},u_{-i})=0 \, (j\in\mc{E}_i), \
 g_{i,j}(x_{i},u_{-i})\ge0 \, (j\in\mc{I}_i).
\end{array}
\right.
\label{eq:checkopt}
\end{equation}
This is a polynomial optimization in $x_i$.
Since $u\in X$, the point $u_i$ is feasible for (\ref{eq:checkopt}),
so $\delta_i\le 0$. If $\delta_i\ge0$ for all $i$,
then $u$ must be a GNE.
The following is an algorithm for solving the GNEP.

\begin{alg}
\label{ag:KKTSDP} \rm
For the convex GNEP given by (\ref{eq:GNEP}),
do the following:

\begin{description}

\item [Step~0]
Choose a generic positive definite matrix $\Theta$ of length $n+s+1$.

\item [Step~1]
Solve the polynomial optimization \reff{eq:KKTfeasopt}.
If it is infeasible, then there are no KKT points and stop;
otherwise, solve it for a minimizer $(u,w)$.

\item [Step~2]
If all $q_i(u) >0$, then $u$ is a GNE. Otherwise,
for those $i$ with $q_i(u) \le 0$, solve the optimization \reff{eq:checkopt}
for the minimum value $\delta_i$.
If $\delta_i \ge 0$ for all such $i$, then $u$ is a GNE;
otherwise, it is not.

\end{description}
\end{alg}

In Step~0, we can choose $\Theta = R^T R$
for a randomly generated square matrix $R$ of length $n+s+1$.
When $\Theta$ is a generic positive definite matrix,
the optimization \reff{eq:KKTfeasopt} must have a unique minimizer,
if its feasible set is nonempty.
This is shown in Theorem~\ref{thmcvg:upconv}(ii).
Since the objective $f_i(x_i,u_{-i})$ is assumed to be convex in $x_i$,
if it is bounded from below on $X_i(u_{-i})$,
then (\ref{eq:checkopt}) must have a minimizer
(see \cite[Theorem~3]{Belousov2000}).
In applications, we are mostly interested in cases that
(\ref{eq:checkopt}) has a minimizer, for the existence of a GNE.
In the subsections~\ref{sc:opt4all} and \ref{sc:opt4one},
we will discuss how to solve polynomial optimization problems in
Algorithm~\ref{ag:KKTSDP},
by the Moment-SOS hierarchy of semidefinite relaxations.
The convergence of Algorithm~\ref{ag:KKTSDP} is shown as follows.

\begin{thm}
\label{tm:KKTSDPconv}
For the convex GNEPP given by (\ref{eq:GNEP}),
suppose each Lagrange multiplier vector $\lambda_i$ has a rational expression
as in Definition~\ref{LME:rat}
or a parametric expression as in Definition~\ref{def:para:LME}.

\begin{enumerate}[(i)]

\item If $(u,w)$ is a feasible point of (\ref{eq:KKTfeasopt})
such that $q_i(u)>0$ for all $i$, then $u$ is a GNE.

\item Assume every GNE is a KKT point.
If (\ref{eq:KKTfeasopt}) is infeasible, then the GNEP has no GNEs.
If $\Theta$ is positive definite and every
$q_i(x)>0$ for all feasible $x$ of (\ref{eq:KKTfeasopt}),
then Algorithm~\ref{ag:KKTSDP} will find a GNE if it exists.
\end{enumerate}
\end{thm}
\begin{proof}
(i) This is directly implied by Propositions~\ref{pr:rtn} and \ref{pr:prt}.

(ii) If (\ref{eq:KKTfeasopt}) is infeasible,
then there is no GNE, because every GNE is assumed to be a KKT point
and it must be feasible for (\ref{eq:KKTfeasopt}).
Next, assume (\ref{eq:KKTfeasopt}) is feasible.
Since $\Theta$ is positive definite,
the optimization (\ref{eq:KKTfeasopt}) has a minimizer,
say, $(u,w)$. By the given assumption,  we have $q_i(u)>0$ for all $i$.
So $u$ is a GNE, by (i).
\end{proof}

\begin{remark}
For convex GNEPs, we can choose not to
use nontrivial expressions for Lagrange multipliers,
i.e., we consider the polynomial optimization (\ref{eq:KKTfeasopt})
with $s_i=m_i$ and $\lambda_{i,j}=\omega_{i,j}$ for all $i$ and $j$.
By doing this, we can get an algorithm like Algorithm~\ref{ag:KKTSDP} to get GNEs.
However, this approach is usually very inefficient computationally,
because it results in more variables for
the polynomial optimization (\ref{eq:KKTfeasopt}).
Note that when Lagrange multiplier expressions (LMEs) are not used,
each Lagrange multiplier is treated as a new variable.
Moreover, solving (\ref{eq:KKTfeasopt}) without LMEs
may require higher order Moment-SOS relaxations.
This is shown in numerical experiments in Section~\ref{sc:opt4all}.
In Example~\ref{ep:fstepinsc6}(i-ii),
we compare the performance of Algorithm~\ref{ag:KKTSDP}
with and without LMEs.
Computational results show the advantage of using them.
\end{remark}

In Theorem~\ref{tm:KKTSDPconv}(ii),
if $q_i(x)>0$ for all $x\in X$, then we must have
$q_i(x)>0$ for all feasible $x$ of (\ref{eq:KKTfeasopt}).
Suppose $(u,w)$ is a computed minimizer of (\ref{eq:KKTfeasopt}).
If $u$ is not a GNE, i.e., $\delta_i<0$ for some $i$,
we can let $\mc{N} \subseteq [N]$ be the labeling set of $i$ with $\delta_i<0$.
By Theorem~\ref{tm:KKTSDPconv}, we know $q_i(u)=0$ for all $i \in \mc{N}$.
For a priori small $\varepsilon >0$,
we can add the inequalities $q_i(x)\ge\varepsilon \, (i \in \mc{N})$
to the optimization (\ref{eq:KKTfeasopt}), to exclude $u$ from the feasible set.
Then we solve the following new optimization
\be
\label{eq:KKTfeasoptadd}
\left\{
\begin{array}{cll}
\min\limits_{\allx\in X\times\re^s} & \lvt\allx\rvt_1^T\Theta\lvt\allx\rvt_1 \\
\st &q_i(x)\nabla_{x_i} f_i(x)-\sum\nolimits_{j=1}^{m_i}\hat{\lambda}_{i,j}(\allx)\nabla_{x_i} g_{i,j}(x)=0 \, (i\in[N]),\\
	&\hat{\lambda}_{i,j}(\allx)\perp g_{i,j}(x)\,(j\in\mathcal{E}_i\cup\mc{I}_i,i\in[N]),\\
	&\hat{\lambda}_{i,j}(\allx)\ge0\,(j\in\mathcal{I}_i,i\in[N]),\\
	&q_i(x)\ge \varepsilon \, (i\in\mc{N}).
\end{array}
\right.
\ee
If $\varepsilon >0$ is not small enough, the constraint $q_i(x)\ge\varepsilon$
may also exclude some GNEs.
If the new optimization (\ref{eq:KKTfeasoptadd}) is infeasible,
one can heuristically get a candidate GNE
by choosing a different generic positive definite $\Theta$ in (\ref{eq:KKTfeasopt}).
In computational practice, when a GNE exists,
it is very likely that we can get one by doing this.
However, how to detect nonexistence of GNEs when
(\ref{eq:KKTfeasopt}) is feasible can be theoretically difficult.
The theoretical side of this problem is mostly open,
to the best of the authors' knowledge.

\subsection{The optimization for all players}
\label{sc:opt4all}

We discuss how to solve the polynomial optimization problems in Algorithm~\ref{ag:KKTSDP},
by using the Moment-SOS hierarchy of semidefinite relaxations
\cite{Las01,lasserre2015introduction,LasICM,Lau09,LauICM}.
We refer to the notation in subsections \ref{ssc:poly} and \ref{ssc:locmat}.

First, we discuss how to solve the optimization \reff{eq:KKTfeasopt}.
Denote the polynomial tuples
\begin{multline}
 \label{poly:Phi:ith}
\Phi_i := \Big \{q_i(x)\nabla_{x_i} f_i(\allx)
  -\sum_{j=1}^{m_i}\hat{\lambda}_{i,j}(\allx)\nabla_{x_i} g_{i,j}(x) \Big \}
\cup \Big \{ g_{i,j}(x): j \in \mc{E}_i \Big  \}  \\
 \cup \Big \{\hat{\lambda}_{i,j}(\allx)\cdot g_{i,j}(x): j \in \mc{I}_i \Big  \} ,
\end{multline}
\begin{multline}
\label{poly:Psi:ith}
\Psi_i : =  \Big \{ g_{i,j}(x): j \in \mc{I}_i \Big  \} \cup
\Big \{\hat{\lambda}_{i,j}(\allx): j \in \mc{I}_i \Big  \}.\qquad\qquad\qquad\qquad\qquad\qquad
\end{multline}
For notational convenience, for a vector $p = (p_1,\ldots, p_s)$,
the set $\{ p \}$ stands for $\{p_1, \ldots, p_s\}$, in the above.
Denote the unions
\[
\Phi \, :=\bigcup_{i=1}^N \Phi_i,  \quad
\Psi \,:=\,\bigcup_{i=1}^N \Psi_i.
\]
They are both finite sets of polynomials.
Then, the optimization (\ref{eq:KKTfeasopt})
can be equivalently written as
\be  \label{eq:rewtupper}
\left\{
\baray{rl}
\vartheta_{\min} :=   \min\limits_{\allx} &
     \theta(\allx) := [\allx]_1^T \Theta [\allx]_1  \\
\st &  p(\allx) = 0 \,\, (\forall \, p \in \Phi),   \\
    &  q(\allx) \ge 0 \,\, (\forall \, q \in \Psi).
\earay
\right.
\ee
Denote the degree
\[
d_0 \, := \, \max\{\lceil\deg(p)/2\rceil: \, p \in \Phi \cup \Psi \}.
\]
For a degree $k \ge d_0$, consider the $k$th order
moment relaxation for solving \reff{eq:rewtupper}
\be
\label{eq:d-mom}
\left\{
\baray{rl}
\vartheta_k  \, :=  \,\min\limits_{ y } & \lip \theta,y \rip\\
 \st & y_0=1,\, L_{p}^{(k)}[y] = 0 \, (p \in \Phi), \\
  & M_k[y] \succeq 0,\, L_{q}^{(k)}[y] \succeq 0 \, (q \in \Psi),  \\
  &  y \in \mathbb{R}^{\mathbb{N}^{n+s}_{2k}}.
\earay
\right.
\ee
Its dual optimization problem is the $k$th order SOS relaxation
\be
\label{eq:d-sos}
\left\{
\baray{ll}
\max & \gamma\\
\st & \theta -\gamma \in  \idl[\Phi]_{2k}+\qmod[\Psi]_{2k}. \\
\earay
\right.
\ee
For relaxation orders $k=d_0, d_0+1, \ldots$,
we get the Moment-SOS hierarchy of semidefinite relaxations
\reff{eq:d-mom}-\reff{eq:d-sos}.
This produces the following algorithm for solving
the polynomial optimization problem \reff{eq:rewtupper}.

\begin{alg} \label{ag:upsdp} \rm
Let $\theta,\Phi,\Psi$ be as in \reff{eq:rewtupper}.
Initialize $k := d_0$.
\begin{description}
\item [Step~1]
Solve the semidefinite relaxation (\ref{eq:d-mom}). If it is infeasible,
then \reff{eq:rewtupper} has no feasible points and stop;
otherwise, solve it for a minimizer $y^*$.

\item [Step~2]
Let $\allu=(u,w): = (y^*_{e_1}, \ldots, y^*_{e_{n+s}})$.
If $\allu$ is feasible for \reff{eq:rewtupper}
and $\vartheta_k = \theta(u)$,
then $\allu$ is a minimizer of \reff{eq:rewtupper}.
Otherwise, let $k := k+1$ and go to Step~1.
\end{description}
\end{alg}

In the Step~2, $e_i$ denotes the labeling vector such that
its $i$th entry is $1$ while all other entries are $0$.
For instance, when $n=s=2$,  $y_{e_3} = y_{0010}$.
The optimization \reff{eq:d-mom} is a relaxation of \reff{eq:rewtupper}.
This is because if $\allx$ is a feasible point of \reff{eq:rewtupper},
then $y = [\allx]_{2k}$ must be feasible for \reff{eq:d-mom}.
Hence, if \reff{eq:d-mom} is infeasible,
then \reff{eq:rewtupper} must be infeasible,
which also implies the nonexistence of KKT points.
Moreover, the optimal value $\vartheta_k$ of \reff{eq:d-mom}
is a lower bound for the minimum value of \reff{eq:rewtupper},
i.e., $\vartheta_k \leq \theta(\allx)$ for all $\allx$
that is feasible for \reff{eq:rewtupper}.
In the Step~2, if $\allu$ is feasible for \reff{eq:rewtupper}
and $\vartheta_k = \theta(\allu)$,
then $\allu$ must be a minimizer of \reff{eq:rewtupper}.
The Algorithm~\ref{ag:upsdp} can be implemented in {\tt GloptiPoly} \cite{GloPol3}.
The convergence of Algorithm~\ref{ag:upsdp} is shown as follows.

\begin{thm} \label{thmcvg:upconv}
Assume the set $\idl[\Phi]+\qmod[\Psi] \subseteq \re[\allx]$ is archimedean.
\bit

\item [(i)]
If \reff{eq:rewtupper} is infeasible,
then the moment relaxation \reff{eq:d-mom}
must be infeasible when the order $k$ is big enough.

\item [(ii)]
Suppose \reff{eq:rewtupper} is feasible and
$\Theta$ is a generic positive definite matrix.
Then \reff{eq:rewtupper} has a unique minimizer.
Let $\allu^{(k)}$ be the point $\allu$ produced in the Step~2 of
Algorithm~\ref{ag:upsdp} in the $k$th loop.
Then $\allu^{(k)}$ converges to the unique minimizer of \reff{eq:rewtupper}.
In particular, if the real zero set of $\Phi$ is finite,
then $\allu^{(k)}$ is the unique minimizer of \reff{eq:rewtupper},
when $k$ is sufficiently large.

\eit
\end{thm}
\begin{proof}
(i) If \reff{eq:rewtupper} is infeasible,  the constant polynomial $-1$
can be viewed as a positive polynomial on the feasible set of \reff{eq:rewtupper}.
Since $\idl[\Phi]+\qmod[\Psi]$ is archimedean,
we have $-1 \in  \idl[\Phi]_{2k}+\qmod[\Psi]_{2k}$,
for $k$ big enough, by the Putinar Positivstellensatz \cite{putinar1993positive}.
For such a big $k$, the SOS relaxation \reff{eq:d-sos} is unbounded from above,
hence the moment relaxation \reff{eq:d-mom} must be infeasible.

(ii) When the optimization \reff{eq:rewtupper} is feasible,
it must have a unique minimizer, say, $\allx^*$.
To see this,
let $\theta$ be defined as in \reff{eq:rewtupper},
$K$ be the feasible set of \reff{eq:rewtupper},
and $\mc{R}_2(K)$ be the set of tms's in $\re^{\N^n_2}$
admitting K-representing measures.
Consider the linear conic optimization problem
\be\label{eq:momopt}
\left\{ \begin{array}{ll}
\min  & \lip\theta, y\rip\\
\st & y_{0}=1, y\in\mc{R}_2(K).
\end{array} \right.
\ee
If $\Theta$ is generic in the cone of positive definite matrice,
the objective $\lip\theta, y\rip$ is a generic linear function in $y$.
By~\cite[Proposition~5.2]{nie2014moment},
the optimization (\ref{eq:momopt}) has a unique minimizer.
The minimum value of (\ref{eq:momopt}) is equal to $\vartheta_{\min}$.
Therefore, (\ref{eq:rewtupper}) has a unique minimizer when $\Theta$ is generic.
The convergence of $\allu^{(k)}$ to $\allx^*$ is shown in
\cite{Swg05} or \cite[Theorem~3.3]{nie2013certifying}.
For the special case that $\Phi(\allx)=0$ has finitely many real solutions,
the point $\allu^{(k)}$ must be equal to $\allx^*$,
when $k$ is large enough.
This is shown in \cite{lasserre2008semidef}
(also see \cite{nie2013polynomial}).
\end{proof}

The archimedeaness of the set $\idl[\Phi]+\qmod[\Psi]$
is essentially requiring that the feasible set of \reff{eq:rewtupper} is compact.
The archimedeaness is sufficient but not necessary for Algorithm~\ref{ag:upsdp} to converge.
Even if the archimedeaness fails to hold,
Algorithm~\ref{ag:upsdp} is still applicable for solving (\ref{eq:KKTfeasopt}).
If the point $\allu^{(k)}$ is feasible and $\vartheta_k=\theta(\allu^{(k)})$,
then $\allu^{(k)}$ must be a minimizer of (\ref{eq:KKTfeasopt}),
regardless of the archimedeaness holds or not.
Moreover, without archimedeaness,
the infeasibility of (\ref{eq:d-mom}) still implies that (\ref{eq:KKTfeasopt}) is infeasible.
In our computational practice,
Algorithm~\ref{ag:upsdp} almost always has finite convergence.

The polynomial optimization (\ref{eq:KKTfeasoptadd})
can be solved in the same way by the Moment-SOS hierarchy of semidefinite relaxations.
The convergence property is the same. For the cleanness of this paper,
we omit the details.

\subsection{Checking Generalized Nash Equilibria}
\label{sc:opt4one}

Suppose $\allu=(u,w)\in\re^n\times\re^s$ is a minimizer of \reff{eq:KKTfeasopt}.
For convex GNEPPs, if all $q_i(u)>0$,
then $u$ is a GNE, by Theorem~\ref{tm:KKTSDPconv}(i).
If $q_i (u) \le 0$ for some $i$,
we need to solve the optimization \reff{eq:checkopt},
to check if $u=(u_i,u_{-i})$ is a GNE or not,
Note that \reff{eq:checkopt} is a convex polynomial optimization problem in $x_i$.
For given $u_{-i}$, if it is bounded from below,
then (\ref{eq:checkopt}) achieves its optimal value at a minimizer.

Consider the $i$th player's optimization with $q_i (u) \le 0$.
For notational convenience, we denote the polynomial tuples
\begin{multline} \label{polylow:Hi(v)}
H_i(u) := \big \{g_{i,j}(x_i,u_{-i}): j \in \mc{E}_i \big \} \cup
\big \{  \hat{\lambda}_{i,j}(\xpi,u_{-i})\cdot g_{i,j}(x_i,u_{-i}): j \in \mc{I}_i \big \}  \\
\cup \big \{ q_i(\xpi,u_{-i})\nabla_{x_i}f_i(\xpi, u_{-i}) -
 \sum_{j=1}^{m_i} \hat{\lambda}_{i,j}(\xpi, u_{-i} ) \nabla_{x_i} g_{i,j}(x_i,u_{-i}) \big \} ,
\end{multline}
\be  \label{polylow:Gi(v)}
J_i(u) := \big \{g_{i,j}(\xpi, u_{-i} ): j \in \mc{I}_i \big \} \cup
\big \{  \hat{\lambda}_{i,j}(\xpi,u_{-i}) : j \in \mc{I}_i \big \}   .
\ee
Like in \reff{poly:Phi:ith}-\reff{poly:Psi:ith},
the set $\{ p \}$ stands for $\{p_1, \ldots, p_s\}$,
when $p = (p_1,\ldots, p_s)$ is a vector of polynomial.
The sets $H_i(u), J_i(u)$ are finite collections.

Under some suitable constraint qualification conditions (e.g., the Slater's Condition),
when (\ref{eq:checkopt}) has a minimizer,
it is equivalent to
\be  \label{eq:rewtlower}
\left\{
\baray{rrl}
\eta_i:= &\min\limits_{x_i \in \re^{n_i} }  &
    \zeta_i(\xpi):=f_i(\xpi, u_{-i})-f_i(u_i, u_{-i}) \\
&\st & p(\xpi)=0 \, (p \in H_i(u) ), \\
    && q(\xpi) \ge 0  \, (q \in J_i(u) ).
\earay
\right.
\ee
Denote the degree in variables $\xpi$ for its constraining polynomials
\begin{multline}   \label{hfdg:di:ith}
d_{i}:=\max \big \{ \lceil \deg(\zeta_i(\xpi, u_{-i}))/2,\deg(p(\xpi))/2, \\
 \deg(q(\xpi))/2: p \in H_i(u), q \in J_i(u)   \rceil\big \}.
\end{multline}
For a degree $k \ge d_{i}$, the $k$th order moment relaxation for
\reff{eq:rewtupper} is
\be
\label{eq:d-momlow}
\left\{
\baray{rrl}
\eta_i^{(k)}:=&\min\limits_{ y } & \lip \zeta_i(\xpi), y \rip\\
&\st & y_0=1,\, L_{p}^{(k)}[y] = 0 \, (p \in H_i(u)),  \\
 & & M_k[y] \succeq 0,\, L_{q}^{(k)}[y] \succeq 0 \, (q \in J_i(u)), \\
 & &  y\in\mathbb{R}^{\mathbb{N}^{n_i}_{2k}}.
\earay
\right.
\ee
The dual optimization problem of (\ref{eq:d-momlow}) is the $k$th order SOS relaxation
\be
\label{eq:d-soslow}
\left\{
\baray{ll}
\max & \gamma\\
\st & \zeta_i(\xpi) - \gamma \in
        \idl[H_i(u)]_{2k}+\qmod[J_i(u)]_{2k}. \\
\earay
\right.
\ee
By solving the above relaxations for $k =d_i, d_i+1, \ldots$,
we get the Moment-SOS hierarchy of relaxations
\reff{eq:d-momlow}-\reff{eq:d-soslow}.
This gives the following algorithm.

\begin{alg} \label{ag:lowsdp} \rm
For a minimizer $u = (u_i, u_{-i})$ of \reff{eq:KKTfeasopt} with $q_i(u) \le 0$,
solve the $i$th player's optimization \reff{eq:rewtlower}. Initialize $k := d_i$.
\begin{description}

\item [Step~1]
Solve the moment relaxation (\ref{eq:d-momlow})
for the minimum value $\eta_i^{(k)}$ and a minimizer $y^*$.
If $\eta_i^{(k)}\ge0$,
then $\eta_i=0$ and stop;
otherwise, go to the next step.

\item [Step~2]
Let $t:=d_i$ as in \reff{hfdg:di:ith}.
If $y^*$ satisfies the rank condition
\be \label{eq:flatranklow}
 \Rank{M_t[y^*]} \,=\, \Rank{M_{t-d_i}[y^*]} ,
\ee
then extract a set $U_i$ of
 $r :=\Rank{M_t(y^*)}$ minimizers for \reff{eq:rewtlower} and stop.

\item [Step~3]
If \reff{eq:flatranklow} fails to hold and $t < k$,
let $t := t+1$ and then go to Step~2;
otherwise, let $k := k+1$ and go to Step~1.

\end{description}
\end{alg}

We would like to remark that the optimization \reff{eq:rewtlower}
is always feasible, because $u_i$ is a feasible point
since $u$ is a minimizer of \reff{eq:KKTfeasopt}.
The moment relaxation (\ref{eq:d-momlow}) is also feasible.
Because $\eta_i^{(k)}$ is a lower bound for $\eta_i$,
and $\eta_i\le \zeta_i(u_i,u_{-i})=0$,
if $\eta_i^{(k)}\ge0$, then $\eta_i$ must be $0$.
In Step~2, the rank condition~\reff{eq:flatranklow} 
is called \textit{flat truncation} \cite{nie2013certifying}.
It is a sufficient (and almost necessary) condition to check convergence
of moment relaxations.
When \reff{eq:flatranklow} holds, the method in \cite{HenLas05}
can be used to extract $r$ minimizers for \reff{eq:rewtlower}.
The Algorithm~\ref{ag:lowsdp} can also be implemented in {\tt GloptiPoly} \cite{GloPol3}.
If $\idl[H_i(u)]+\qmod[J_i(u)]$ is archimedean,
then $\eta_i^{(k)}\to\eta_i$ as $k\to\infty$ \cite{Las01}.
It is interesting to remark that
\[
I_1:=\idl[g_{i,j}(x_i,u_{-i}): j \in \mc{E}_i] \subseteq \idl[H_i(u)],
\]
\[
I_2 := \qmod[g_{i,j}(x_i,u_{-i}): j \in \mc{I}_i] \subseteq \qmod[J_i(u)] .
\]
If $I_1 + I_2$ is archimedean, then $\idl[H_i(u)]+\qmod[J_i(u)]$
must also be archimedean.
Furthermore, we have the following convergence theorem
for Algorithm~\ref{ag:lowsdp}.

\begin{remark}
To check the flat truncation (\ref{eq:flatranklow}),
we need to evaluate the ranks of $M_t[y^*]$ and $M_{t-d_i}[y^*]$.
Evaluating matrix ranks is a classical problem in numerical linear algebra.
When a matrix is near to be singular, it may be difficult to
determine its rank accurately, due to round-off errors.
In computational practice, we often determine the rank of a matrix
as the number of its singular values larger than a tolerance (say, $10^{-6}$).
We refer to \cite{Demmel} for determining matrix ranks numerically.
Moreover, when \reff{eq:rewtlower} has a unique minimizer,
the ranks of $M_t[y^*]$ and $M_{t-d_i}[y^*]$ are one,
the flat truncation (\ref{eq:flatranklow})
is relatively easy to check by looking at the largest singular value.
\end{remark}

\begin{thm}  \label{thmcvg:lowopt}
For the convex polynomial optimization \reff{eq:checkopt},
assume its optimal value is achieved at a KKT point.
If either one of the following conditions hold,
\bit

\item [(i)]
The set $I_1+I_2$ is archimedean, and the Hessian
$\nabla_{x_i}^2\zeta_i(x_i^*,u_{-i}) \succ 0$
for a minimizer $x_i^*$ of \reff{eq:rewtlower};
{\rm\textbf{or}}
\item [(ii)]
The real zero set of polynomials in $H_i(u)$ is finite,
\eit
then Algorithm~\ref{ag:lowsdp} must terminate within
finitely many loops.
\end{thm}
\begin{proof}
Since its optimal value is achieved at a KKT point,
the optimization problem~\reff{eq:checkopt}
is equivalent to \reff{eq:rewtlower}.

(i)
If $I_1+I_2$ is archimedean and $\nabla_{x_i}^2\zeta_i(x_i^*,u_{-i})\succ0$ if
$x_i^*$ is a minimizer of \reff{eq:rewtlower},
then $\zeta_i(x_i)-\eta_i\in I_1+I_2$, by \cite[Corollary~3.3]{Klerk2011}.
Since \[ I_1+I_2\subseteq \idl[H_i(u)]+\qmod[J_i(u)], \]
we have $\zeta_i(x_i)-\eta_i\in \idl[H_i(u)]_{2k}+\qmod[J_i(u)]_{2k}$
for all $k$ big enough.
Therefore, Algorithm~\ref{ag:lowsdp}
must terminate within finitely many loops, by the duality theory.

(ii)
If the real zero set of polynomials in $H_i(u)$ is finite,
then the conclusion is implied by \cite[Theorem~1.1]{nie2013polynomial} and \cite[Theorem~2.2]{nie2013certifying}.
\end{proof}

\begin{remark}
If the objective polynomial in (\ref{eq:checkopt})
is SOS-convex and its constraining ones are SOS-concave
(see \cite{Helton2010} for the definition of SOS-convex polynomials),
then Algorithm~\ref{ag:lowsdp} must terminate in the first loop (see \cite{Las09}).
If the optimal value of \reff{eq:checkopt} is not achieved at a KKT point,
the classical Moment-SOS hierarchy of semidefinite relaxations
can be used to solve it. We refer to
\cite{Klerk2011,Las09,Las01,lasserre2015introduction,LasICM,Lau09,LauICM}
for the work for solving general polynomial optimization.
\end{remark}

\section{Numerical experiments}
\label{sc:ne}

In this section, we apply Algorithm~\ref{ag:KKTSDP} to solve convex GNEPs.
To use it, we need Lagrange multiplier expressions.
This can be done as follows.

\begin{itemize}

\item When polynomial expressions exist, we always use them.
In particular, we use polynomial expressions for the first player
of the GNEP given by \reff{eq:firstep},
the first player in Example~\ref{ep:fstepinsc6}(ii),
the third player in Examples~\ref{ep:A8} and \ref{ep:hybrid}(i-ii),
the production unit and market players in Example~\ref{ep:market}.

\item We use rational expressions for all players in
Examples~\ref{eq:rational_example_insc6}, \ref{ep:jointballinsc6} and \ref{ep:num_lmd}.
Moreover, rational expressions are used for the second player of the GNEP given by \reff{eq:firstep}, the first two players in Examples~\ref{ep:A8}
and \ref{ep:hybrid}(i-ii),
and the consumer players in Example~\ref{ep:market}.
For Example~\ref{ep:num_lmd}, the rational expression is obtained by
solving (\ref{eq:findL}) numerically.

\item
When it is difficult to find convenient polynomial or rational expressions,
we use parametric expressions for Lagrange multipliers.
For all players in Examples~\ref{ep:partialepinsc6}, \ref{ep:A3},
we use parametric expressions.
\end{itemize}

We apply the software {\tt GloptiPoly~3} \cite{GloPol3}
and {\tt SeDuMi} \cite{sturm1999using} to solve
the Moment-SOS relaxations for the polynomial optimization
(\ref{eq:rewtupper}) and (\ref{eq:rewtlower}).
We use the software {\tt YALMIP} for solving (\ref{eq:findL}).
The computation is implemented in an Alienware Aurora R8 desktop,
with an Intel\textsuperscript \textregistered \
Core(TM) i7-9700 CPU at 3.00GHz$\times $8 and 16GB of RAM,
in a Windows 10 operating system.
For neatness of the paper, only four decimal digits are shown
for the computational results.

In Step~2 of Algorithm~\ref{ag:KKTSDP},
if the optimal values $\delta_i \ge 0$ for each $i$ such that $q_i(u)\le0$,
then the computed minimizer of (\ref{eq:KKTfeasopt}) is a GNE.
In numerical computations,
we may not have $\delta_i \ge 0$ exactly due to round-off errors.
Typically, when $\delta_i$ is near zero, say, $\delta_i \ge -10^{-6}$,
we regard the computed solution as an accurate GNE.
In the following, all the GNEPs are convex.

\begin{exm}\rm
\label{ep:fstepinsc6}
(i) For the GNEP given by \reff{eq:firstep},
the first player has a polynomial expression for Lagrange multipliers given by (\ref{eq:1stlmd}),
and the second player has a rational expression given as
\[
\lambda_{2,1}=\frac{-x_2^T\nabla_{x_2}f_2}{2q_2(x)},\quad q_2(x)=x_1^Tx_1.
\]
For each $i$, the $q_i(x)>0$ for all $x\in X$.
We ran Algorithm~\ref{ag:KKTSDP} and obtained the GNE $u = (u_1, u_2)$ with
\be \label{eq:1stGNE}
u_1 \approx (0.7274,0.7274,0.7274),\quad
u_2 \approx (0.4582,0.4582,0.4582).
\ee
It took around $2.83$ seconds.\\
(ii) If the first player's objective is changed to
\[
f_1(x) = (x_{2,1}+x_{2,2}-2x_{2,3})(x_{1,1}+x_{1,2}-2x_{1,3})^2+x_{1,1}+x_{1,2}-2x_{1,3},
\]
then the GNEP has no GNE, detected by Algorithm~\ref{ag:KKTSDP}.
It took around $70.31$ seconds to detect the nonexistence.
The matrix polynomials $G_1(x)$ and $G_2(x)$ are nonsingular on $X$,
so all GNEs must be KKT points if they exist.
\end{exm}

In the following, we compare the performance of Algorithm~\ref{ag:KKTSDP}
with the method of solving the optimization \reff{eq:KKTfeasopt}
without using Lagrange multiplier expressions,
i.e., each Lagrange multiplier is treated as a new variable for polynomials.
The comparison for Example~\ref{ep:fstepinsc6}(i)
is given in Table~\ref{tab:ep6.1_hassol}.
The computational results for the method using Lagrange multiplier expressions
(i.e., for Algorithm~\ref{ag:KKTSDP})
are given in the column labeled ``Algorithm~\ref{ag:KKTSDP}".
The results for the method without using Lagrange multiplier expressions
are given in the column labeled ``Without LME".
In the rows, the value $k$ is the relaxation order for solving (\ref{eq:KKTfeasopt}).
The subcolumn ``time'' lists the consumed time (in seconds)
for solving the moment relaxation of order $k$,
and the subcolumn ``GNE" shows if a GNE is obtained or not.
When $k=2$ for Algorithm~\ref{ag:KKTSDP},
the degree of relaxation is less than appearing polynomials,
so we display that ``not applicable (n.a.)".
\begin{table}[htb]
\centering
\caption{The computational results for Example~\ref{ep:fstepinsc6}(i).}
\label{tab:ep6.1_hassol}
\begin{tabu}{|c|c|c|c|c|}  \hline
      & \multicolumn{2}{|c|}{Algorithm~\ref{ag:KKTSDP}}  &
             \multicolumn{2}{|c|}{Without LME} \\ \hline
      & time & GNE  & time & GNE \\ \hline
$k=2$ &  n.a. & n.a.  & 4.03 & no \\ \hline
$k=3$ & 2.83 & yes  & 1350.09 & no \\ \hline
\end{tabu}
\end{table}
\noindent
For Example~\ref{ep:fstepinsc6}(ii),
the comparison is given in Table~\ref{tab:ep6.1(ii)}.
This GNEP does not have a GNE. When no LMEs are used,
the $4$th order moment relaxation cannot be solved due to out of memory.
However this can be done by using nontrivial LMEs.
\begin{table}[htbp]
\centering
\caption{The computational results for Example~\ref{ep:fstepinsc6}(ii).}
\label{tab:ep6.1(ii)}
\begin{tabu}{|c|c|c|c|c|}  \hline
      & \multicolumn{2}{|c|}{Algorithm~\ref{ag:KKTSDP}}  &
                \multicolumn{2}{|c|}{Without LME} \\ \hline
      & time & nonexistence of GNE  & time & nonexistence of GNE \\ \hline
$k=2$ &  n.a. & n.a.  & 3.67 & not detected \\ \hline
$k=3$ &  2.77 & not detected  & 1201.75 & not detected \\ \hline
$k=4$ &  70.31 & detected  & \multicolumn{2}{|c|}{out of memory} \\ \hline
\end{tabu}
\end{table}

\begin{exm}\rm
\label{ep:A8insec6}
Consider the GNEP in Example~\ref{ep:A8}.
We use Lagrange multiplier expressions given there.
By Algorithm~\ref{ag:KKTSDP},
we obtained a feasible point $\hat{u}\approx10^{-4}\cdot ( 0.1274, 0.4102, 0.3219)$ of (\ref{eq:KKTfeasopt}) with
$q_1(\hat{u})\approx 0.1274\cdot10^{-4}$ and $q_2(\hat{u})\approx0.4102\cdot10^{-4}$.
We solved (\ref{eq:checkopt}), for $i=1,2$, to check if $\hat{u}$ is a GNE or not,
and got $\delta_1\approx -1.0000$, $\delta_2\approx -1.8996\cdot 10^{-10}$.
Therefore, we solved (\ref{eq:KKTfeasoptadd}) with $\mc{N}=\{1\}$ and $\varepsilon=0.1$,
and obtained a GNE $u = (u_1, u_2, u_3)$ with
\[
u_1 \approx  0.5000,\ u_2 \approx 0.5000,\ u_3 \approx 0.7500, \ q_1(u)\approx q_2(u)\approx0.1250.
\]
It took around $0.89$ second.
\end{exm}

\begin{exm}\rm
\label{eq:rational_example_insc6}
Consider the GNEP in Example~\ref{ep:>0} with objectives
\[
f_1(x)=\sum\limits_{j=1}^2(x_{1,j}-1)^2+x_2(x_{1,1}-x_{1,2}),\
f_2(x)=(x_2)^3-x_{1,1}x_{1,2}x_2-x_2.
\]
The rational expressions for both players are given by (\ref{eq:lmd_>0}).
For each $i$, the $q_i(x)>0$ for all $x\in X$.
We ran Algorithm~\ref{ag:KKTSDP} and got the GNE $u = (u_1, u_2)$ with
\[
u_1 \approx (0.4897,1.0259), \, u_2 \approx 0.7077.
\]
It took around 0.20 second.
\end{exm}

\begin{exm}\rm
\label{ep:jointballinsc6}
Consider the GNEP in Example~\ref{ep:jointball} with objectives
\[
f_1(x)=10x_1^Tx_2-\sum_{j=1}^3x_{1,j},\
f_2(x)= \sum_{j=1}^3(x_{1,j}x_{2,j})^2+(3\prod_{j=1}^3x_{1,j}-1)\sum_{j=1}^3x_{2,j}.
\]
We use rational expressions as in (\ref{eq:lmd_ball}).
From Example~\ref{ep:jointball},
we know all feasible points of (\ref{eq:KKTfeasopt}) are GNEs.
By Algorithm~\ref{ag:KKTSDP}, we got the GNE $u = (u_1, u_2)$ with
\[
u_1 \approx (0.9864,0.0088,0.0088),\quad
u_2 \approx (0.0836,0.0999,0.0999).
\]
It took around 2.03 seconds.
\end{exm}

\begin{exm}\rm
\label{ep:partialepinsc6}
Consider the GNEP in Example~\ref{ep:partialep} with objectives
\[
\begin{array}{l}
f_1(x)=x_{2,1}(x_{1,1})^3+(x_{1,2})^3 -\sum\nolimits_{j=1}^2x_{1,j}\cdot \sum\nolimits_{j=1}^2x_{2,j},\\
f_2(x)= (x_{1,1}+x_{1,2})(x_{2,1})^3-3x_{2,1}+(x_{2,2})^2+x_{1,1}x_{1,2}x_{2,2}.
\end{array}
\]
We use parametric expressions as in (\ref{eq:lmd_partep}).
For each $i$, the $q_i(x)>0$ for all $x\in X$.
By Algorithm~\ref{ag:KKTSDP}, we got the GNE $u = (u_1, u_2)$ with
\[
u_1 \approx (0.6475,0.2786), \quad u_2 \approx (1.0391,-0.0902).
\]
It took around 63.97 seconds. 
\end{exm}

\begin{exm} \rm  \label{ep:num_lmd}
Consider the $2$-player GNEP
\[
\begin{array}{cllcl}
    \min\limits_{x_{1} \in \re^2 }& (x_{1,1})^2+2(x_{1,2})^2 &\vline&
    \min\limits_{x_{2} \in \re^2 }& \Vert x_1\Vert^2\cdot\Vert x_2\Vert^2+3x_1^Tx_2 \\
    &\qquad\quad+3\sum\nolimits_{j=1}^2x_{1,j}(x_{2,j})^2&\vline&
    &\qquad\qquad\qquad\qquad +x_{2,1}-x_{2,2}\\
    \st & x_{1,1}+2x_{1,2}-x_{2,1}\le1,&\vline& \st &(x_{2,1})^2+x_{1,2}x_{2,1}\le2,\\
    &(x_{1,2})^2+(x_{2,1})^2\le3,   &\vline &&(x_{1,1})^2+(x_{2,2})^2\le3,\\
    &x_{1,1}\ge0,&\vline& &x_{2,2}\ge0.\\
\end{array}
\]
We solve (\ref{eq:findL}) numerically for $i=1,2$ with $ v =(0,0,0,0), d=2$
to get rational expressions for $\lambda_i$'s.
By Algorithm~\ref{ag:KKTSDP}, we got the GNE $u = (u_1, u_2)$ with
\[
\begin{array}{ll}
u_1 \approx (0.0000,-1.3758), & \quad q_1(u) \approx 6.7538;\\
u_2 \approx (-0.2641,1.3544), & \quad q_2(u) \approx 2.3227.
\end{array}
\] 
It took around $0.41$ second in solving $(\ref{eq:findL})$ for both players,
and $6.40$ seconds to find the GNE. For neatness of the paper,
we do not display Lagrange multiplier expressions obtained by solving (\ref{eq:findL}).
\end{exm}

\begin{exm}\rm
\label{ep:hybrid}
(i) Consider the 3-player GNEP
\[
\mbox{1st player:}
\left\{\begin{array}{ll}
\min\limits_{x_1\in\re^2} &
x_{2,1}(x_{1,1})^2+x_{2,2}(x_{1,2})^2-(x_{3,1})^2x_{1,1}-(x_{3,2})^2x_{1,2}\\
\st&x_1^Tx_1\le1+x_2^Tx_2;
\end{array}\right.
\]
\[
\mbox{2nd player:}
\left\{\begin{array}{ll}
\min\limits_{x_2\in\re^2} &
(x_{2,1})^3+(x_{2,2})^3-x_{1,1}x_{2,1}x_{3,1}-x_{1,2}x_{2,2}x_{3,2}\\
\st&x_{2,1}+x_{2,2}\le 1+x_3^Tx_3, \, x_{2,1}\ge0,\, x_{2,2}\ge0;\\ 
\end{array}\right.
\]
\[
\mbox{3rd player:}
\left\{\begin{array}{ll}
\min\limits_{x_3\in\re^2} &
\big(\sum_{i=1}^3(x_{i,1}+x_{i,2}) \big)^2-x_{3,1}-x_{3,2}\\
\st&x_{3,1}\ge x_{1,1}, \, x_{3,2}\ge x_{1,2}.
\end{array}\right.\]
The first player's Lagrange multipliers have a rational expression,
that
\[
\lambda_1=\frac{-x_1^T\nabla_{x_1}f_1}{2q_1(x)},\quad
q_1(x)=1+x_2^Tx_2.
\]
For the second player, we let $q_2(x)=1+x_3^Tx_3$,
and there exists a rational expression for $\lambda_2$, that
\[
\lambda_{2,1}=\frac{-x_2^T\nabla_{x_2}f_2}{q_2(x)},\quad
\lambda_{2,2}=\frac{\partial f_2}{\partial x_{2,1}}+\lambda_{2,1},\quad
\lambda_{2,3}=\frac{\partial f_2}{\partial x_{2,2}}+\lambda_{2,1}.
\] 
For $\lambda_3$,
we use the polynomial expression that
\[
\lambda_{3,1}=\frac{\partial f_3}{\partial x_{3,1}},\quad
\lambda_{3,2}=\frac{\partial f_3}{\partial x_{3,2}}.
\]
For each $i$, the $q_i(x)>0$ for all $x\in X$.
By Algorithm~\ref{ag:KKTSDP}, we got the GNE $u = (u_1, u_2, u_3)$ with
\[
\begin{array}{c}
u_1 \approx (0.1097,0.0750),\
u_2 \approx (0.0663,0.0458),\
u_3 \approx (0.1205,0.0828). 
\end{array}
\] 
It took around $3.23$ seconds. \\ 
\noindent
(ii) If the third player's objective function becomes
\[
\baray{c}
\big(\sum_{i=1}^3 (x_{i,1}-x_{i,2}) \big)^2-x_{3,1}-x_{3,2},
\earay
\]
then Algorithm~\ref{ag:KKTSDP} took around $2.86$
seconds to detect nonexistence of GNEs.
Note that all the matrix polynomials $G_i(x)\,(i=1,\dots,3)$ are nonsingular on $X$,
so all GNEs must be KKT points if they exist.
\end{exm}

\begin{exm}\rm
\cite[Example A.3]{FacKan10}
\label{ep:A3}
Consider the GNEP of $3$ players. 
For $i=1,2,3$, the $i$th player aims to minimize the quadratic function
$$f_i(x)=\frac{1}{2}x_i^TA_ix_i+x_i^T(B_i\xmi+b_i).$$
All variables have box constraints $-10\le x_{i,j}\le 10$, for all $i,j$.
In addition to them, the first player has linear constraints $x_{1,1}+x_{1,2}+x_{1,3}\le 20,\,x_{1,1}+x_{1,2}-x_{1,3}\le x_{2,1}-x_{3,2}+5$;
the second player has $x_{2,1}-x_{2,2}\le x_{1,2}+x_{1,3}-x_{3,1}+7$;
and the third player has $x_{3,2}\le x_{1,1}+x_{1,3}-x_{2,1}+4.$
The values of parameters are set as follows
\[
\begin{array}{cccc}
A_1=\lvt\begin{array}{ccc}
20&5&3\\5&5&-5\\3&-5&15
\end{array}\rvt,\
A_2=\lvt\begin{array}{ccc}
11&-1\\-1&9
\end{array}\rvt,\
A_3=\lvt\begin{array}{ccc}
48&39\\39&53
\end{array}\rvt,\\
B_1=\lvt\begin{array}{cccc}
-6&10&11&20\\10&-4&-17&9\\15&8&-22&21
\end{array}\rvt,\
B_2=\lvt\begin{array}{ccccc}
20&1&-3&12&1\\10&-4&8&16&21
\end{array}\rvt,\\
B_3=\lvt\begin{array}{ccccc}
10&-2&22&12&  16  \\
9&19&21&-4&20
\end{array}\rvt,\
b_1=\lvt\begin{array}{ccc}
1\\-1\\1
\end{array}\rvt,\
b_2=\lvt\begin{array}{ccc}
1\\0
\end{array}\rvt,\
b_3=\lvt\begin{array}{ccc}
-1\\2
\end{array}\rvt.
\end{array}\]
We use parametric expressions for Lagrange multipliers as in \reff{eq:boxlmd}.
It is clear $q_i(x)=1$ for all $x\in X$ and for all $i=1,2,3.$
By Algorithm~\ref{ag:KKTSDP}, we got the GNE $u = (u_1, u_2, u_3)$ with
\[
\begin{array}{c}
u_1 \approx (-0.3805,-0.1227,-0.9932), \quad
u_2 \approx (0.3903,1.1638),\\ u_3 \approx (0.0504,0.0176).
\end{array}\]
It took around 8.18 seconds.
\end{exm}

\begin{exm}\rm
\label{ep:market}
Consider the GNEP based on the Arrow and Debreu model of a competitive economy \cite{arrow1954existence, FacKan10}.
The first $N_1$ players are consumers,
the second $N_2$ players are production units,
and the last player is the market, so $N=N_1+N_2+1$.
In this GNEP, each player has $n_1=\dots=n_N$ variables.
Let $Q_i\in\re^{n_i\times n_i},b_i\in\re^{n_i},\xi_i\in\re^{n_i}_+$ and $a_{i,k}\in\re_+$ be parameters.
These players' optimization problems are:
\[\mbox{The $i$th player (a consumer):}\left\{
\begin{array}{ll}
\min\limits_{x_i\in\re^{n_i}_+}&\frac{1}{2}x_i^TQ_ix_i-b_i^Tx_i\\
\st   & x_{N}^Tx_i\le x_{N}^T\xi_i+\sum_{k=N_1+1}^{N-1}a_{i,k}x_{N}^Tx_k .
\end{array}\right.\]
\[\mbox{The $i$th player (a production unit):}\left\{
\begin{array}{ll}
\min\limits_{x_i\in\re^{n_i}_+}&-x_{N}^Tx_i\\
\st   &  x_i^Tx_i\le i-N_1 .
\end{array}\right. \qquad \qquad \qquad \]
\[\mbox{The $N$th player (the market):}\left\{
\begin{array}{ll}
\min\limits_{x_N\in\re^{n_i}_+}&x_N^T\left(
\sum_{k=N_1+1}^{N-1}x_k-\sum_{k=1}^{N_1}(x_k-\xi_k)\right)\\
\st   & \sum_{j=1}^{{n_i}}x_{N,j}=1 .
\end{array}\right.\]
For each $i\in[N_1]$, the Lagrange multipliers have rational expressions as
\[
\lambda_{i,1}=\frac{-x_i^T\nabla_{x_i}f_i}{q_i(x)},\
\lambda_{i,j}=\frac{\partial f_i}{\partial x_{i,j}}+x_{N,j}\cdot
\lambda_{i,1}\,(j=1,\ldots,{n_i}),
\]
where $q_i(x)=x_{N}^T\xi_i+\sum_{k=N_1+1}^{N-1}a_{i,k}x_{N}^Tx_k>0$ for all $x\in X$.
For each $i = N_1+1, \ldots, N_1+N_2$,
the $i$th player (a production unit) has polynomial expressions
\[
\lambda_{i,1}=\frac{-x_i^T\nabla_{x_i}f_i}{2(i-N_1)},\ \lambda_{i,j}
=\frac{\partial f_i}{\partial x_{i,j}}+2x_{i,j}\cdot\lambda_{i,1}\,(j=1,\ldots,n_i).
\]
For the last player (the market),
we substitute $x_{N,n_i}$ by $1-\sum_{j=1}^{n_i-1}x_{N,j}$,
then the constraints become
$1-\sum_{j=1}^{n_i-1}x_{N,j}\ge0,\, x_{N,1}\ge0,\ldots, x_{N,n_i-1}\ge0,$ and hence
\[
\lambda_{N,1}=-\sum\nolimits_{j=1}^{n_i-1}\frac{\partial f_N}{\partial x_{N,j}}\cdot x_{N,j},\ \lambda_{N,j+1}=\frac{\partial f_N}{\partial x_{N,j}}+\lambda_{N,1}\,(j=1,\ldots,n_i-1).
\]
For each $i=1,\dots,N_1$, when $n_i=2$, the parameters are given as
\[\begin{array}{lccc}
Q_i=\lvt\begin{array}{ccc}
0.75+0.25i&1.5-0.5i\\ 1.5-0.5i &i
\end{array}\rvt,\quad
b_i=\lvt\begin{array}{ccc}
0.4+0.1i\\0.9+0.1i
\end{array}\rvt, \quad \xi_i=\lvt\begin{array}{ccc}
i\\i
\end{array}\rvt, \\
a_{i,j}=0.2+0.1i\quad (j=N_1+1,\dots, N_1+N_2).
\end{array}\]
When $n_i=3$,
the parameters are given as:
\[\begin{array}{lccc}
Q_i=\lvt\begin{array}{ccc}
-1+2i&-i&i\\-i&1+i&1-i\\i&1-i&1+i
\end{array}\rvt,\quad
b_i=\lvt\begin{array}{ccc}
0.4+0.1i\\0.9+0.1i\\ 1.4+0.1i
\end{array}\rvt, \quad \xi_i=\lvt\begin{array}{ccc}
i\\i\\i
\end{array}\rvt, \\
a_{i,j}=0.2+0.1i\quad (j=N_1+1,\dots, N_1+N_2).
\end{array}\]
The numerical results are presented in Table~\ref{tab:market}.
The ``$N$'' is the total number of all players,
the ``$N_1$'' and ``$N_2$'' are the number of consumers and production units respectively,
the ``$n$'' (resp., ``$n_i$'') is the dimension of ``$x$'' (resp., ``$x_i$''),
the ``$u$" is the GNE obtained by Algorithm~\ref{ag:KKTSDP},
the ``$q(u)$" gives the value of the
denominator vector $q(u):=(q_1(u) \ddd q_{N_1}(u))$,
and ``time'' shows the consumed time (in seconds).
\begin{table}[htb]
\renewcommand{\arraystretch}{1}
\centering
\caption{Numerical results of Example~\ref{ep:market}}
\label{tab:market}
\begin{tabular}{|c|c|c|c|c|c|c|c|c|c|c|}  \hline
$\begin{array}{c}\mbox{Number}\\\mbox{of}\\\mbox{players} \end{array}$ & dimension & $u$ & $q(u)$ & time \\ \hline
$\begin{array}{l}N\ =5\\N_1=2\\N_2=2 \end{array}$ & $\begin{array}{l}n\ =10\\n_i=2 \end{array}$ & $\begin{array}{l}
(0.0000,1.0000,\quad 0.2889,0.4778,\\ \ 0.4166,0.9091,\quad 0.5892,1.2856,\\ \qquad\qquad  0.3143,0.6857)\end{array}$ & $\begin{array}{c} 1.4907\\2.6543\end{array}$ & 1.37\\ \hline
$\begin{array}{l}N\ =6\\N_1=3\\N_2=2 \end{array}$ & $\begin{array}{l}n\ =12\\n_i=2 \end{array}$ & $\begin{array}{l}
(0.0000,1.0000,\quad 0.2889,0.4778,\\ \ 0.4667,0.4000,\quad 0.4354,0.9002,\\ \ 0.6157,1.2731,\quad 0.3260,0.6740)\end{array}$ & $\begin{array}{c}1.5423\\2.7230\\3.9038\end{array}$ & 3.82\\ \hline
$\begin{array}{l}N\ =7\\N_1=3\\N_2=3 \end{array}$ & $\begin{array}{l}n\ =14\\n_i=2 \end{array}$& $\begin{array}{l}
(0.0000,1.0000,\quad 0.2889,0.4778,\\ \ 0.4667,0.4000,\quad 0.5587,0.8294,\\ \ 0.7901,1.1729,\quad 0.9677,1.4365,\\ \qquad\qquad 0.4025,0.5975)\end{array}$ &$\begin{array}{c} 1.8961\\3.1948\\4.4935\end{array}$ & 21.26\\ \hline
$\begin{array}{l}N\ =8\\N_1=4\\N_2=3 \end{array}$ & $\begin{array}{l}n\ =16\\n_i=2 \end{array}$& $\begin{array}{l}
(0.0000,1.0000,\quad 0.2889,0.4778,\\ \ 0.4667,0.4000,\quad 0.5704,0.3963,\\ \ 0.5835,0.8121,\quad 0.8251,1.1485,\\ \ 1.0106,1.4067,\quad 0.4181,0.5819)\end{array}$ & $\begin{array}{c} 1.8913\\3.1884\\4.4855\\5.7826\end{array} $ & 106.78\\ \hline
$\begin{array}{l}N\ =9\\N_1=4\\N_2=4 \end{array}$ & $\begin{array}{l}n\ =18\\n_i=2 \end{array}$& $\begin{array}{l}
(0.0000,1.0000,\quad 0.2889,0.4778,\\ \ 0.4667,0.4000,\quad 0.5704,0.3963,\\ \ 0.6258,0.7800,\quad 0.8850,1.1031,\\ \ 1.0838,1.3510,\quad 1.2515,1.5600,\\ \qquad\qquad 0.4451,0.5549)\end{array}$ & $\begin{array}{c}
2.3116\\3.7489\\5.1861\\6.6233
\end{array} $ & 465.71\\ \hline
$\begin{array}{l}N\ =3\\N_1=1\\N_2=1 \end{array}$ & $\begin{array}{l}n\ =9\\n_i=3 \end{array}$& $\begin{array}{l}
(1.3076,1.0871,0.0962, \\ \ 0.8087,0.5882,0.0000, \\ \ 0.5789
,0.4211,0.0000)
\end{array}$ & $1.2148$ & 0.50\\ \hline
$\begin{array}{l}N\ =4\\N_1=2\\N_2=1 \end{array}$ & $\begin{array}{l}n\ =12\\n_i=3 \end{array}$& $\begin{array}{l}
(1.3696,1.0886,0.0652, \\ \ 0.3500,0.7875,0.5625, \\ \ 0.6245,0.7810,0.0000, \\ \ 0.4443,0.5557,0.0000)
\end{array}$ & $\begin{array}{c}1.2134\\2.2846\end{array}$ & 3.76\\ \hline
$\begin{array}{l}N\ =5\\N_1=2\\N_2=2 \end{array}$ & $\begin{array}{l}n\ =15\\n_i=3 \end{array}$& $\begin{array}{l}
(1.7172,1.3109,0.0000, \\ \ 0.3500,0.7875,0.5625, \\ \ 0.7006,0.7135,0.0097, \\ \ 0.9908,1.0091,0.0006, \\ \ 0.4953,0.5044,0.0003)
\end{array}$ & $\begin{array}{c}1.5121\\2.6829\end{array}$ &42.66\\ \hline
$\begin{array}{l}N\ =6\\N_1=3\\N_2=2 \end{array}$ & $\begin{array}{l}n\ =18\\n_i=3 \end{array}$& $\begin{array}{l}
(1.7734,1.3398,0.0000, \\ \ 0.3500,0.7875,0.5625, \\ \ 0.2250,0.7958,0.6542, \\ \ 0.5780,0.8160,0.0001, \\ \ 0.8174,1.1541,0.0040,\\ \ 0.4146,0.5854,0.0000)\end{array}$ & $\begin{array}{c}1.5192\\2.6923\\3.8653\end{array}$ & 473.84\\ \hline
\end{tabular}

\end{table}
\end{exm}

\subsection{Comparison with other methods}

We compare our method (i.e., Algorithm~\ref{ag:KKTSDP})
with some classical methods for solving convex GNEPPs, such as
the two-step method in \cite{Han2012} based on Quasi-variational formulation,
the penalty method in \cite{FacKan10},
the exact version of interior point method
based on the KKT system in \cite{dreves2011solution},
and the Augmented-Lagrangian method in \cite{kanzow2016}.
All examples in Section~\ref{sc:ne} are tested for comparisions.
For Example~\ref{ep:market}, we test for the case that $N_1=N_2=1,n_i=3$.

For a computed tuple $u := (u_1, \ldots, u_N)$, we use the value
\[
\xi \,:= \, \max \big \{\max_{i\in[N],j\in\mc{I}_i}\{-g_{i,j}(u)\},
\max_{i\in[N],j\in\mc{E}_i}\{|g_{i,j}(u)|\} \big\}
\]
to measure the feasibility violation.
Clearly, the point $u$ is feasible if and only if $\xi \le 0$.
If we solve (\ref{eq:checkopt}) for all $i\in[N]$,
the accuracy parameter of $u$ is $\delta:=\max_{i\in[N]}\vert\delta_i\vert$.
For these methods, we use the following stopping criterion:
For each time we get a new iterate $u$,
if its feasibility violation $\xi <10^{-6}$,
then we compute the accuracy parameter $\dt$.
If $\delta<10^{-6}$, then we stop the iteration.

For these classical methods, the parameters are the same as given in \cite{Han2012,FacKan10,dreves2011solution,vonHeusinger2009,kanzow2016}.
When implementing the QVI method, 
we use Moment-SOS relaxations 
to find projections into given sets
(the maximum number of iterations for line search is set to be $100$).
For the penalty method, the MATLAB function {\tt fsolve}
is used to implement the Levenberg-Marquardt Algorithm
for solving all equations involved
(the maximum number of iterations is set to be $100$).
The full penalization is used when we implement the Augmented-Lagrangian method,
and a Levenberg-Marquardt type method (see \cite[Algorithm~24]{kanzow2016})
is exploited to solve penalized subproblems.
We let $1000$ be the maximum number of iterations for the QVI method, 
let $1000$ be the maximum number of outer iterations
for the penalty method and the Augmented-Lagrangian method,
and let $10,000$ be the maximum number of iterations for the interior point method.
For initial points, we use $(1,0,0,1,0,0)$ for Example~\ref{ep:fstepinsc6}(i-ii),
$(0,0,0,0,0,0,0,0,1)$ for Example~\ref{ep:market},
and the zero vectors for other GNEPs.
If the maximum number of iterations is reached
but the stopping criterion is not met,
we still solve the (\ref{eq:checkopt}) to check if
the latest iterating point is a GNE or not.

\begin{table}[htbp]
\renewcommand{\arraystretch}{1.20}
\centering
\caption{Comparison with some methods}
\label{tab:comparison}
\begin{tabular}{|c|c|c|c|c|c|c|c|c|c|c|c|}  \hline
\multicolumn{2}{|c|}{Example}         & QVI  & Penalty     & IPM & A-L & Alogrithm~\ref{ag:KKTSDP} \\ \hline
\multirow{2}{*}{\ref{ep:fstepinsc6}(i)} & time &\multirow{2}{*}{Fail} & \multirow{2}{*}{Fail} & \multirow{2}{*}{Fail}  & \multirow{2}{*}{Fail}  & 2.83 \\ \cline{2-2}\cline{7-7}
& error & & & & & $4\cdot 10^{-9}$ \\ \hline
\multirow{2}{*}{\ref{ep:fstepinsc6}(ii)}
& time &\multirow{2}{*}{Fail} & \multirow{2}{*}{Fail} & \multirow{2}{*}{Fail}  & \multirow{2}{*}{Fail}  & 70.31 \\ \cline{2-2}\cline{7-7}
& error & & & & & no GNE \\ \hline
\multirow{2}{*}{\ref{ep:A8insec6}}
& time  &\multirow{2}{*}{Fail} & 3.45 & 0.19 & \multirow{2}{*}{Fail}  & 0.89 \\ \cline{2-2}\cline{4-5}\cline{7-7}
& error & & $2\cdot 10^{-6}$ & $3\cdot 10^{-7}$ & & $7\cdot 10^{-7}$ \\ \hline
\multirow{2}{*}{\ref{eq:rational_example_insc6}}
& time  & 2.63 & 8.46 & 0.12 & 0.08 & 0.20 \\ \cline{2-7}
& error & $8\cdot 10^{-7}$ & $3\cdot 10^{-6}$ & $2\cdot 10^{-7}$ & $2\cdot 10^{-7}$ & $1\cdot 10^{-8}$ \\ \hline
\multirow{2}{*}{\ref{ep:jointballinsc6}}
& time  & \multirow{2}{*}{Fail} & 4.51 & 0.29 & \multirow{2}{*}{Fail} & 2.03 \\ \cline{2-2}\cline{4-5}\cline{7-7}
& error &  & $3\cdot 10^{-5}$ & $8\cdot 10^{-7}$ &  & $4\cdot 10^{-7}$ \\ \hline
\multirow{2}{*}{\ref{ep:partialepinsc6}}
& time  & 185.29 & 4.02 & 37.7 & 0.03 & 63.97 \\ \cline{2-7}
& error & $9\cdot 10^{-5}$ & $2\cdot 10^{-6}$ & $5\cdot 10^{-4}$ &$3\cdot 10^{-7}$ & $4\cdot 10^{-7}$ \\ \hline
\multirow{2}{*}{\ref{ep:num_lmd}}
& time  & 7.78 & \multirow{2}{*}{Fail} & 0.17 & \multirow{2}{*}{Fail} & 6.40 \\ \cline{2-3}\cline{5-5}\cline{7-7}
& error & $6\cdot 10^{-7}$ &  & $3\cdot 10^{-7}$ & & $1\cdot 10^{-7}$ \\ \hline
\multirow{2}{*}{\ref{ep:hybrid}(i)}
& time  & 72.18 & 0.39 & 0.16 & 0.05 & 3.23 \\ \cline{2-7}
& error & $4\cdot 10^{-7}$ & $8\cdot 10^{-8}$ & $5\cdot 10^{-7}$ & $1\cdot 10^{-10}$ & $7\cdot 10^{-9}$ \\ \hline
\multirow{2}{*}{\ref{ep:hybrid}(ii)}
& time &\multirow{2}{*}{Fail} & \multirow{2}{*}{Fail} & \multirow{2}{*}{Fail}  & \multirow{2}{*}{Fail}  & 2.86 \\ \cline{2-2}\cline{7-7}
& error & & & & & no GNE \\ \hline
\multirow{2}{*}{\ref{ep:A3}}
& time &\multirow{2}{*}{Fail} & 0.38 & 0.16  & 0.01  & 8.18 \\ \cline{2-2}\cline{4-7}
& error & & $9\cdot 10^{-8}$ & $1\cdot 10^{-8}$ & $1\cdot 10^{-8}$ & $3\cdot 10^{-8}$ \\ \hline
\multirow{2}{*}{\ref{ep:market}}
& time &1.223 & 6.26 & 0.14  & \multirow{2}{*}{Fail}  & 0.50 \\ \cline{2-2}\cline{3-5}\cline{7-7}
& error & $3\cdot 10^{-5}$ & $8\cdot 10^{-6}$ & $3\cdot 10^{-7}$ &  & $7\cdot 10^{-7}$ \\ \hline 
\end{tabular}%
\end{table}

The numerical results are presented in Table~\ref{tab:comparison},
and the comparison is summarized in the following.

\begin{enumerate}

\item The QVI method failed to find a GNE for Example~\ref{ep:fstepinsc6}(i),
because the projection set in Step~2 is empty.
Therefore the line-search could not finish (see \cite[Algorithm~4.1]{Han2012}).
This is also the case for
Examples~\ref{ep:fstepinsc6}(ii) and \ref{ep:hybrid}(ii),
for which the GNEs do not exist.
For Examples~\ref{ep:A8insec6} and \ref{ep:jointballinsc6},
the sequence generated by QVI is alternating
between several points and none of them is a GNE.
For Example~\ref{ep:A3}, the sequence does not converge.

\item The penalty method failed to find a GNE for
Examples~\ref{ep:fstepinsc6}(i) and \ref{ep:num_lmd},
because the equation $F_{\varepsilon_k}(x)=0$
cannot be solved for some $k$ (see \cite[Algorithm~3.3]{FacKan10}).
This is also the case for
Examples~\ref{ep:fstepinsc6}(ii) and \ref{ep:hybrid}(ii),
for which the GNEs do not exist.

\item The interior-point method failed to find a GNE for
Examples~\ref{ep:fstepinsc6}(i), \ref{ep:fstepinsc6}(ii) and \ref{ep:hybrid}(ii),
because the step-length is too small to efficiently decrease the violation of KKT conditions.
Note that for Examples~\ref{ep:fstepinsc6}(ii) and \ref{ep:hybrid}(ii),
the GNEs do not exist, so the Newton type directions usually
do not satisfy the sufficient descent conditions.

\item The Augmented-Lagrangian method failed to find a GNE for
Example~\ref{ep:fstepinsc6}(i),
because the maximum penalty parameter ($10^{12}$) is reached
before a GNE is obtained.
This is also the case for Example~\ref{ep:fstepinsc6}(ii),
for which the GNEs do not exist.
For Examples~\ref{ep:A8insec6}, \ref{ep:jointballinsc6}, \ref{ep:num_lmd},
\ref{ep:hybrid}(ii) and \ref{ep:market},
the Augmented-Lagrangian method failed to find a GNE,
because the penalization subproblems cannot be efficiently solved.

\end{enumerate}

\section{Conclusions and Discussions}
\label{sc:conc}

This paper studies convex GNEPs given by polynomials.
The rational and parametric expressions for Lagrange multipliers are used.
Based on these expressions, Algorithms~\ref{ag:KKTSDP} is proposed for computing a GNE.
The Moment-SOS hierarchy of semidefinite relaxations are used to
solve the appearing polynomial optimization problems.
Under some general assumptions, we show that Algorithm~\ref{ag:KKTSDP}
is able to find a GNE if there exists one,
or detect nonexistence of GNEs if there is none.

For future work, it is interesting to solve nonconvex GNEPPs.
Under some constraint qualifications, the KKT system (\ref{eq:KKTwithLMalp})
is necessary but not sufficient for GNEs.
A solution $u$ of (\ref{eq:KKTwithLMalp}) may not be a GNE for nonconvex GNEPPs.
If $u$ is not a GNE,
one needs to find an efficient method to obtain a different candidate.
Such a method is proposed for solving NEPs \cite{Nie2020nash}.
For GNEPs, it is not clear how to generalize the method in \cite{Nie2020nash}.
When the point $u$ is not a GNE,
how can we exclude it and find a better candidate?
When (\ref{eq:KKTfeasopt}) is feasible,
how do we detect nonexistence of GNEs?
These questions are mostly open, to the best of the authors' knowledge.

\bigskip\noindent
{\bf Acknowledgements}
The authors would like to thank Christian Kanzow and Daniel Steck for sharing
the code for solving GNEPs.
They also thank the editors and anonymous referees for fruitful suggestions.

\end{document}